\documentclass{article}
\usepackage[margin=1in]{geometry}
\usepackage{amsmath, amssymb}
\usepackage{amsthm}
\usepackage{amsmath}
\usepackage[hidelinks]{hyperref}
\numberwithin{equation}{section}
\usepackage{amsthm,amsmath,amsfonts,amssymb,xcolor,amsxtra,appendix,bookmark,dsfont,bm,mathrsfs}
\usepackage{algorithm}
\usepackage{algpseudocode}
\usepackage{graphicx}
\usepackage{subfig}

\theoremstyle{plain}
\newtheorem{theorem}{Theorem}
\newtheorem{Lemma}{Lemma}
\newtheorem{Remark}{Remark}
\newtheorem{Problem}{Problem}
\newtheorem{Definition}{Definition}
\newtheorem{Assumption}{Assumption}
\newtheorem{corollary}{Corollary}
\newtheorem{proposition}{Proposition}

\title{A globally convergent Carleman--Picard method for an inverse initial-value problem for a nonlinear diffusive coagulation--fragmentation equation}

\author{
       Thuy T. Le\thanks{Department of Mathematics and Statistics, California State University, Long Beach, CA 90032, USA, \texttt{Thuy.Le@csulb.edu}.}  
     \and
     Minh-Binh Tran\thanks{Department of Mathematics, Texas A\&M University, College Station, TX 77843, USA, \texttt{minhbinh@tamu.edu}}
     \and
        Loc H. Nguyen\thanks{Department of Mathematics and Statistics, University of North Carolina at Charlotte, NC 28223, USA, \texttt{loc.nguyen@charlotte.edu}}
    }

\begin{document}
\date{}

\maketitle

\begin{abstract}
We study an inverse initial-density problem for a nonlinear diffusive coagulation--fragmentation equation with known coagulation and fragmentation kernels. The objective is to recover the unknown initial particle-size distribution on a finite interval from time-dependent boundary observations of the solution and its size derivative. To solve this inverse problem, we develop a globally convergent numerical method based on a Legendre--exponential time reduction and a Carleman--Picard iteration. The time reduction transforms the original problem into a nonlinear coupled system for the spatial mode coefficients, while the Carleman weight and the corresponding Carleman estimate guarantee the global convergence of the Picard iteration without requiring a good initial guess. We prove the convergence of the proposed method and obtain a complete reconstruction procedure for the initial density. Numerical experiments with noisy boundary data demonstrate that the method yields accurate and stable reconstructions for several representative test profiles.
\end{abstract}


\section{Introduction}

Let $T>0$ be a final time, and let
\[
f:(0,\infty)\times(0,T)\to\mathbb{R}
\]
denote the density of particles of size $v$ at time $t$.
In coagulation-fragmentational models, the density $f$ evolves under the combined effects of coagulation, fragmentation, convection in the size variable, and diffusion in the size variable.
In our work, those effects are incorporated into the following coagulation-fragmentation equation with size convection-diffusion
\cite{laurenccot2022well,matioc2025quasilinear, olesen2005diffusion}:
\begin{equation}\label{coagfrag}
	\begin{cases}
		\partial_t f(v,t) = -b(v)\partial_v f(v,t) + \partial_{vv} f(v,t) + Q(f)(v,t),
		& (v,t)\in (0,\infty)\times (0,T),\\
		f(0,t)=0, & t\in(0,T),\\
		f(v,0)=f^0(v), & v\in(0,\infty),
	\end{cases}
\end{equation}
where $f^0$ denotes the initial particle-size distribution. The operator $Q(f)$ is decomposed as
\begin{equation}\label{Collision}
	Q(f)(v,t)=Q_{\mathrm{coag}}(f)(v,t)+Q_{\mathrm{frag}}(f)(v,t),
\end{equation}
where
\begin{equation}\label{coag}
	Q_{\mathrm{coag}}(f)(v,t)
	= \frac{1}{2}\int_{0}^{v} K(v-v^\ast,v^\ast)\, f(v-v^\ast,t)\,f(v^\ast,t)\,dv^\ast
	- f(v,t)\int_{0}^{\infty} K(v,v^\ast)\, f(v^\ast,t)\,dv^\ast,
\end{equation}
and
\begin{equation}\label{frag}
	Q_{\mathrm{frag}}(f)(v,t)
	= - f(v,t)\int_{0}^{v} V(v-v^\ast,v)\,dv^\ast
	+ 2\int_{0}^{\infty} V(v,v^\ast)\,f(v+v^\ast,t)\,dv^\ast.
\end{equation}
The microscopic structure of the system is encoded in the nonnegative kernels
\[
K(v,v^\ast)\ge 0, \qquad V(v,v^\ast)\ge 0,
\]
which are assumed to be symmetric in $v$ and $v^\ast$. The kernel $K$ describes the coagulation rate at which two clusters of sizes $v$ and $v^\ast$ merge to form a single cluster of size $v+v^\ast$, whereas $V$ characterizes the fragmentation rate at which a cluster of size $v+v^\ast$ breaks into two clusters of sizes $v$ and $v^\ast$. Interactions with the surrounding medium, allowing for the exchange of monomeric units, are modeled by the size-convective term $-b(v)\partial_v f(v,t)$ and the size-diffusive term $\partial_{vv}f(v,t)$.

Equation \eqref{coagfrag} and related coagulation--fragmentation models have been extensively investigated in both theoretical and numerical settings; see, for instance, \cite{BCG13, BallCarr, Banasiak1, Bertoin1, CCM, DubovskiStewart, jang2025discrete, Melzak, MLM, PerthameRyzhik, Stewart90, tran2022coagulation, Montroll_Simha40, ZiffMcGrady85} and the references therein. For broader surveys and additional references, we refer the reader to \cite{Banasiak1, Banasiak2, zora77208}. 

In this paper, we propose a globally convergent method for an inverse initial-value problem associated with \eqref{coagfrag}. Let $L>0$ be fixed, and assume that the unknown initial density $f^0$ is supported in $[0,L]$, that is, $f^0(v)=0$ for all $v>L$. We assume that the coagulation and fragmentation kernels are known. Using time-dependent boundary observations of the solution and its size derivative at $v=0$ and $v=L$, we aim to reconstruct $f^0$ on $[0, L]$.

\begin{Problem}[Inverse initial-density problem]\label{prob:inverse_initial_density}
Let $f$ be a sufficiently smooth solution of \eqref{coagfrag}--\eqref{frag}. Given the time-dependent boundary observations
\begin{equation}\label{obs}
\phi_0(t):=f(0,t),\qquad \phi_L(t):=f(L,t),\qquad
\psi_0(t):=\partial_v f(0,t),\qquad \psi_L(t):=\partial_v f(L,t),
\end{equation}
for $t \in (0, T)$,
determine the unknown initial density $f^0$ on $[0,L]$.
\end{Problem}

This inverse problem is of practical interest because direct measurement of the full particle density $f(v,t)$ over the entire size--time domain is generally expensive and difficult to implement. In many realistic situations, it is far more feasible to collect data only at a few accessible observation points, such as the boundary locations $v=0$ and $v=L$. If the initial density $f^0$ can be successfully reconstructed from such limited measurements, then it can be extended by zero for $v>L$, and the full evolution of the particle system can subsequently be computed by solving the forward problem \eqref{coagfrag} with existing analytical \cite{Banasiak2} or numerical methods \cite{das2025numerical,filbet2004numerical}. In this way, one can recover the particle density across the entire domain while significantly reducing experimental cost and measurement effort.

Problem \ref{prob:inverse_initial_density} is extremely challenging. One source of difficulty is the nonlinear collision operator $Q(f)$, whose complicated integro-differential structure makes the forward map from the initial density to the boundary observations highly nonlinear. In principle, one may apply a conventional least-squares approach together with Tikhonov regularization. However, the resulting cost functional is generally nonconvex and may possess multiple local minima. Therefore, such an approach is only locally convergent and typically requires a good initial guess, which is often unavailable in practice. In addition, inverse problems with boundary measurements are usually ill-posed, meaning that small noise in the data can lead to large reconstruction errors. This severe instability poses an additional obstacle to the reliable recovery of $f^0$.
To overcome these difficulties, we develop a globally convergent reconstruction framework for Problem \ref{prob:inverse_initial_density} based on two main ingredients: a Legendre--exponential time reduction and a Carleman--Picard iterative procedure. The global convergence of the method is guaranteed by the use of a Carleman weight, together with suitable Carleman estimates, which provide the key mechanism for controlling the reconstruction process without requiring a good initial guess. On the other hand, the ill-posedness of the inverse problem is mitigated by truncating the Fourier expansion of the data with respect to the polynomial--exponential basis, thereby filtering out highly oscillatory noise components. As a result, the original inverse problem is reduced to a finite, coupled system of equations for the mode coefficients, which can then be solved stably and effectively.

The idea of combining time reduction with Carleman estimates was first introduced in \cite{LeNguyen:jiip2022}, where an inverse initial-value problem for a quasilinear parabolic equation was solved. Later, in \cite{Nguyen:AVM2023}, it was observed that the approach developed in \cite{LeNguyen:jiip2022} can in fact be interpreted as the construction of a contraction mapping whose fixed point is the desired solution. As a consequence, the associated Picard iteration converges globally, even when the initial guess is far from the true solution. Since then, this framework has been extended to a variety of inverse problems; see, for example, \cite{AbneyLeNguyenPeters, LeCON2023, LeNguyenNguyenPark, NguyenNguyenVu2026, DangNguyenVu, VanLeNguyen}. 
The Carleman--Picard method was further generalized to nonlinear inverse problems for hyperbolic equations in \cite{LeKlibanov:ip2022, NguyenKlibanov:ip2022}. In particular, the approach developed in \cite{LeKlibanov:ip2022} was shown to apply to experimental data. 
However, these earlier methods are not directly applicable to Problem \ref{prob:inverse_initial_density} because of the strong nonlinearity and complicated integro-differential structure arising from the coagulation and fragmentation effects. The main contribution of the present paper is to develop a Carleman--Picard method tailored to this coagulation--fragmentation model.

We now briefly review the literature most relevant to the present work.
Since the 1970s, inverse problems for coagulation--fragmentation models have attracted considerable attention in the broader scientific community \cite{NRamG80, Ram74}. Most existing mathematical studies are concerned with inverse problems for linear (fragmentation-only) or linearized versions of the model; see, e.g., \cite{AgoshkovDubovski,Alomari,BCG13,BCGM,Bertoin,CCGU,DET18,doumic2024inverse,HNRT19,KK05,Norris}. Moreover, the data used in these works are typically internal measurements rather than boundary observations of the type considered in Problem \ref{prob:inverse_initial_density}. Inverse results for nonlinear coagulation--fragmentation models are much more limited, and the existing approaches also appear to rely on internal data (see, for instance \cite{agoshkov2002solution,friedman1990hyperbolic,PhysRevE.88.012138,wright1990inverse,zaks2025fast}). To the best of our knowledge, we have not found any work on inverse problems for full coagulation--fragmentation equations that uses only boundary measurements as in our setting.

%
%

The remainder of the paper is organized as follows. In Section \ref{sec2}, we present the main analytical ingredients used throughout the paper, including a one-dimensional Carleman estimate, the Legendre--exponential basis for time-dimensional reduction, and Lipschitz estimates for the projected coagulation and fragmentation operators. In Section~\ref{sec:time_reduction}, we eliminate the time variable and derive a reduced coupled system of ordinary differential equations for the expansion coefficients. In Section \ref{sec:carleman_picard}, we introduce the Carleman--Picard iteration for solving the reduced inverse problem and prove its global convergence. In Section~\ref{sec5}, we describe the numerical implementation of the method, explain the generation of synthetic boundary data, and present several numerical experiments to validate the proposed reconstruction procedure. Finally, Section \ref{sec6} contains the concluding remarks.

\section{Carleman estimate, Legendre--exponential basis, and Lipschitz estimates}\label{sec2}

 In this section, we present several analytical ingredients needed to design a numerical solver for Problem~\ref{prob:inverse_initial_density}. We first present a one-dimensional Carleman estimate. We then present the Legendre--exponential basis used for the time-dimensional reduction and introduce the projected system. Finally, we establish Lipschitz estimates for the projected coagulation and fragmentation operators on the admissible set. The Carleman estimate and these Lipschitz bounds play a crucial role in proving the global convergence of the Carleman--Picard method developed in Section~\ref{sec:carleman_picard}.

\subsection{A one-dimensional Carleman estimate}\label{sec:carleman_1d}

We summarize a one-dimensional Carleman estimate, which will be the main analytical tool in the Carleman--Picard method developed later. Fix $L>0$ and choose $v_0<0$. Define
$
r(v):=v-v_0, v\in[0,L],
$
so that $r(v)>0$ on $[0,L]$. For parameters $\lambda>0$ and $\beta>0$, we introduce the Carleman weight $e^{2\lambda r(v)^{-\beta}}$, which is used to weight the energy terms in the Carleman estimate stated below. The estimate below provides weighted control of $u$ and its derivatives in terms of $u''$ (up to boundary terms), and will be used to ensure stability and convergence of our reconstruction scheme.

\begin{Lemma}[Carleman estimate in $1$D]\label{thm:carleman_1d_pointwise}
There exists a constant $\beta_0>0$ such that for every $\beta\ge \beta_0$ there exists
$\lambda_0>0$ (depending on $L$, $v_0$, and $\beta$) with the following property: for all
$\lambda\ge \lambda_0$ and all $u\in C^2([0,L])$, one has the pointwise estimate
\begin{equation}\label{eq:carleman_pointwise_1d}
r(v)^{\beta+2}e^{2\lambda r(v)^{-\beta}}|u''(v)|^2
\ge
C\Big(
U'(v)
+\lambda^3\beta^4 e^{2\lambda r(v)^{-\beta}} r(v)^{-2\beta-2}|u(v)|^2 
+\lambda\beta\, e^{2\lambda r(v)^{-\beta}}|u'(v)|^2
\Big),
\end{equation}
for $v \in (0, L)$,
where $C>0$ is independent of $u$ and $(\lambda,\beta)$, and the auxiliary function $U$ satisfies
\begin{equation}\label{eq:U_bound_1d}
|U(v)|
\le
C\,e^{2\lambda r(v)^{-\beta}}
\Big(\lambda^3\beta^3\, r(v)^{-2\beta-2}|u(v)|^2 + \lambda\beta\,|u'(v)|^2\Big),
\qquad v\in(0,L).
\end{equation}
\end{Lemma}

\begin{Remark}
The estimate in Lemma~\ref{thm:carleman_1d_pointwise} is the $1$D counterpart of the piecewise Carleman estimate established in \cite[Section~3]{LeLeNguyen:2024} (corresponding to the case $d=1$ and $A\equiv 1$). For this reason, we omit the proof and refer the reader to that reference.
\end{Remark}

Integrating \eqref{eq:carleman_pointwise_1d} over $[0, L]$ and using \eqref{eq:U_bound_1d}, we obtain the following integrated form, which includes boundary terms at $v=0$ and $v=L$.

\begin{corollary}[Integrated Carleman estimate]\label{cor:carleman_1d_integral}
Under the assumptions of Lemma~\ref{thm:carleman_1d_pointwise}, there exists $C>0$ such that
\begin{multline}\label{eq:carleman_integral_1d}
\int_{0}^{L} e^{2\lambda r(v)^{-\beta}} |u''(v)|^2\,dv
\;\ge\;
C\int_{0}^{L} e^{2\lambda r(v)^{-\beta}}
\Big(\lambda^3|u(v)|^2+\lambda|u'(v)|^2\Big)\,dv
\\
-\,C\sum_{\xi\in\{0,L\}}
e^{2\lambda r(\xi)^{-\beta}}
\Big(\lambda^3|u(\xi)|^2+\lambda|u'(\xi)|^2\Big).
\end{multline}
In particular, if $u(0)=u(L)=u'(0)=u'(L)=0$, then the boundary contribution in \eqref{eq:carleman_integral_1d} vanishes.
\end{corollary}

\subsection{The Legendre polynomial--exponential basis}
\label{sec:legendre_review}

The Legendre polynomial--exponential basis, first introduced in \cite{TrongElastic}, is fundamental to the time-dimensional reduction method employed in this study. This basis combines the spectral structure of the classical Legendre polynomials with an exponential weight, yielding an orthonormal system in an exponentially weighted Hilbert space. For the reader's convenience, we briefly summarize the main definitions and properties below.

Let $\{P_n\}_{n\ge 0}$ denote the Legendre polynomials on $(-1,1)$, defined by Rodrigues' formula
\[
P_n(x)=\frac{1}{2^n n!}\frac{d^n}{dx^n}(x^2-1)^n.
\]
Using the affine transformation
\[
x=\frac{2t}{T}-1, \qquad t\in(0,T),
\]
we define the rescaled polynomials
\[
Q_n(t):=\sqrt{\frac{2n+1}{T}}\,P_n\!\left(\frac{2t}{T}-1\right), \qquad t\in(0,T).
\]
The family $\{Q_n\}_{n\ge 0}$ forms an orthonormal basis of $L^2(0,T)$. We then define the Legendre polynomial--exponential basis by
\[
\Psi_n(t):=e^tQ_n(t), \qquad t\in(0,T), \quad n\ge 0.
\]
The system $\{\Psi_n\}_{n\ge 0}$ forms an orthonormal basis of $L^2(0, T)$ with respect to the inner product
\[
\langle u,v\rangle_{e^{-2t}}
:=
\int_0^T e^{-2t}u(t)v(t)\,dt.
\]

\begin{proposition}[See \cite{TrongElastic}]\label{prop:basis-properties}
The Legendre polynomial--exponential basis functions $\Psi_n$, $n\ge 0$, satisfy the following properties.
\begin{enumerate}
    \item For each $n\ge 0$, the function $\Psi_n$ is infinitely differentiable on $(0, T)$, and none of its derivatives of any order vanishes identically on this interval.
    
    \item For every integer $\ell\in\mathbb{N}$, there exists a constant $C>0$, depending only on $\ell$ and $T$, such that for all $u\in H^\ell(0,T)$,
    \begin{equation}\label{eq:coeff_decay}
        \sum_{n=0}^{\infty} n^{2\ell}
        \left|
        \langle u,\Psi_n\rangle_{e^{-2t}}
        \right|^2
        \le C\|u\|_{H^\ell(0,T)}^2.
    \end{equation}
    
    \item There exists a constant $C>0$, depending only on $T$, such that for all $n\ge 1$,
    \begin{equation}\label{eq:psi_prime_growth}
        \|\Psi_n'\|_{e^{-2t}} \le C n^{3/2}.
    \end{equation}
\end{enumerate}
\end{proposition}

\begin{Remark}[The role of the weight $e^t$]\label{rem2}
We omit the proof of Proposition~\ref{prop:basis-properties}, since it follows from Proposition~2.1, Lemma~2.1, and the first part of Lemma~2.2 in \cite{TrongElastic}. The exponential factor $e^t$ in the definition $\Psi_n=e^tQ_n$ plays a crucial role. Without this weight, some derivatives of the time basis functions may vanish identically; for instance, this happens for the constant mode. See item 2 of Remark~\ref{rem7} for the significance of this property in our numerical method.
\end{Remark}

The following proposition is the one-derivative analogue of the convergence result proved in \cite{TrongElastic} for the second derivative.

\begin{proposition}\label{prop:first_derivative_basis}
Let $p\ge 0$ and assume that
\[
u\in H^\ell\big((0,T);H^p(0, L)\big)\qquad \text{for some }\ell\ge 3.
\]
Denote the Legendre--exponential coefficients of $u$ by
\[
u_n(\cdot):=\left\langle u(\cdot,\cdot),\Psi_n\right\rangle_{L^2_{e^{-2t}}(0,T)}
=\int_0^T e^{-2t}u(\cdot,t)\Psi_n(t)\,dt,\qquad n\ge 0.
\]
Then $u_t\in L^2\big((0,T);H^p(0, L)\big)$ and
\begin{equation*}
\partial_t u(\cdot,t)=\sum_{n=0}^{\infty} u_n(\cdot)\,\Psi_n'(t)
\quad \text{in } L^2\big((0,T);H^p(0, L)\big).
\end{equation*}
\end{proposition}

\begin{Remark}
The proof of Proposition~\ref{prop:first_derivative_basis} follows the same arguments as the proof of the second-derivative result in \cite{TrongElastic}. The threshold $\ell\ge 3$ arises from combining the coefficient decay estimate \eqref{eq:coeff_decay} with the derivative bound \eqref{eq:psi_prime_growth}, see \cite[Theorem 1]{VanLeNguyen} for more details.
\end{Remark}

\subsection{Some Lipschitz continuities}

Let $N$ be a cutoff number and let $\{\Psi_n\}_{n=0}^N$ be the first $N+1$ elements of the
Legendre--exponential basis introduced above. To simplify notation, we introduce the projected operator in the coefficient space.

\begin{Definition}[The projected collision operators]\label{def:projected_Q}
For
$
{\bf f}=\begin{bmatrix} f_0 & \dots & f_N \end{bmatrix}^\top \in [H^3(0,\infty)]^{N+1},
$
define, for each $m=0,\dots,N$ and $v\in(0,\infty)$,
\begin{align*}
Q_{\mathrm{coag},m}({\bf f})(v)
&:=
\int_0^T
Q_{\mathrm{coag}}\!\Big(\sum_{n=0}^N f_n(v)\Psi_n(t)\Big)\,e^{-2t}\Psi_m(t)\,dt,
\\
Q_{\mathrm{frag},m}({\bf f})(v)
&:=
\int_0^T
Q_{\mathrm{frag}}\!\Big(\sum_{n=0}^N f_n(v)\Psi_n(t)\Big)\,e^{-2t}\Psi_m(t)\,dt,
\\
Q_m({\bf f})(v)
&:= Q_{\mathrm{coag},m}({\bf f})(v)+Q_{\mathrm{frag},m}({\bf f})(v).
\end{align*}
In vector form, we write
\[
{\bf Q}_{\mathrm{coag}}({\bf f})(v):=
\begin{bmatrix}
Q_{\mathrm{coag},0}({\bf f})(v)\\
\vdots\\
Q_{\mathrm{coag},N}({\bf f})(v)
\end{bmatrix},
\quad
{\bf Q}_{\mathrm{frag}}({\bf f})(v):=
\begin{bmatrix}
Q_{\mathrm{frag},0}({\bf f})(v)\\
\vdots\\
Q_{\mathrm{frag},N}({\bf f})(v)
\end{bmatrix},
\quad
{\bf Q}({\bf f})(v):=
\begin{bmatrix}
Q_0({\bf f})(v)\\
\vdots\\
Q_N({\bf f})(v)
\end{bmatrix},
\]
for all $v\in[0,\infty)$.
\end{Definition}

In Definition~\ref{def:projected_Q}, the projected operators
$Q_{\mathrm{coag},m}({\bf f})$, $Q_{\mathrm{frag},m}({\bf f})$, and $Q_m({\bf f})$
are defined for coefficient vectors
\[
{\bf f}\in [H^3(0,\infty)]^{N+1}.
\]
However, in our inverse problem, the unknown coefficient vector ${\bf f}$ is sought only on the computational interval $(0, L)$. Therefore, in order to evaluate these projected operators for $v\in(0,L)$, we must extend ${\bf f}$ from $(0,L)$ to $(0,\infty)$.

\begin{Remark}
Note that, when numerically solving coagulation--fragmentation models \cite{das2025numerical,filbet2004numerical,tran2026analysis}, and more generally kinetic equations posed on the full space  \cite{filbet2004numerical,filbet2006solving,gamba2009spectral}, it is necessary to truncate the computational domain to a bounded interval. Accordingly, in the inverse problem considered above, the observation domain in \(v\) is restricted from \((0,\infty)\) to \((0, L)\). Owing to this truncation, the measurements may contain noise near the boundary. Our method, however, remains robust to such boundary perturbations, and noisy boundary data does not pose any essential difficulty for the numerical results.
\end{Remark}

\begin{Remark}[Exponential tail extension]
Throughout this subsection, whenever a coefficient vector
\[
{\bf f}=\begin{bmatrix}f_0 & \dots & f_N\end{bmatrix}^\top
\]
is only prescribed on $(0,L)$, we extend each component to $(0,\infty)$ by
\begin{equation}\label{eq:mode_extension}
f_n(v):=f_n(L)e^{-(v-L)},\qquad v\ge L,\quad n=0,\dots,N.
\end{equation}
All occurrences of $Q_{\mathrm{coag},m}({\bf f})$, $Q_{\mathrm{frag},m}({\bf f})$, and $Q_m({\bf f})$ in this subsection are understood
with this extension.
\end{Remark}

To control the nonlinear projected operators defined in Definition \ref{def:projected_Q} and to obtain uniform estimates in the subsequent analysis, we restrict
attention to coefficient vectors that satisfy an a priori bound. Such a restriction is standard in nonlinear inverse problems and
allows the Lipschitz constants in our estimates to depend only on the prescribed bound. We therefore introduce the following
admissible set.

\begin{Definition}[Admissible set]\label{def:admissible_set}
Let $M\gg 1$ be a prescribed number. Define
\begin{equation*}
B:=\Big\{{\bf f}\in[H^3(0,L)]^{N+1}:\ 
\|{\bf f}\|_{[L^2(0,L)]^{N+1}}+\|{\bf f}\|_{[L^\infty(0,L)]^{N+1}} \le M
\Big\}.
\end{equation*}
\end{Definition}

\begin{Assumption}[Coefficient conditions]\label{ass:kernel_conditions}
The kernels $K$ and $V$ are continuous functions that satisfy
\begin{equation}\label{eq:kernel_bounds}
\sup_{v\in(0,L)}\int_0^\infty K(v,v^\ast)^2 e^{-2(v^\ast-L)_+}\,dv^\ast <\infty,
\qquad
\sup_{v\in(0,L)}\int_0^\infty V(v,v^\ast)^2 e^{-2v^\ast}\,dv^\ast <\infty,
\end{equation}
and
\begin{equation}\label{eq:kernel_tail_bounds}
\sup_{v\in(0,L)}\int_0^\infty K(v,v^\ast)e^{-(v^\ast-L)_+}\,dv^\ast<\infty,
\qquad
\sup_{v\in(0,L)}\int_0^\infty V(v,v^\ast)e^{-v^\ast}\,dv^\ast<\infty.
\end{equation}
In addition, \(K\) and \(V\) are symmetric, that is,
\[
K(v^\ast,v)=K(v,v^\ast)
\quad \text{and} \quad
V(v^\ast,v)=V(v,v^\ast)
\quad \text{for all } v,v^\ast\ge 0.
\]
Moreover, \(b\in L^\infty(0,\infty)\).
\end{Assumption}

\begin{Remark} Some examples of \(K\) are given below:
\begin{itemize}
	\item[(i)] \(K(v,v^\ast)=c_K (v+v^\ast)^a\), where \(c_K>0\) and \(a\ge 0\).
	
	\item[(ii)] \(K(v,v^\ast)=c_K \bigl(v^a+(v^\ast)^a\bigr)\), where \(c_K>0\) and \(a\ge 0\).
	
	\item[(iii)] \(K(v,v^\ast)=c_K (v+v^\ast)^a\bigl(v^b+(v^\ast)^b\bigr)\), where \(c_K>0\) and \(a,b\ge 0\).
	
	\item[(iv)] \(K(v,v^\ast)=c_K \bigl(v^a+(v^\ast)^a\bigr)\bigl(v^b+(v^\ast)^b\bigr)\), where \(c_K>0\) and \(a,b\ge 0\).
	
	\item[(v)] \(K(v,v^\ast)=c_K \bigl(v^a(v^\ast)^b+(v^\ast)^a v^b\bigr)\), where \(c_K>0\) and \(a,b\ge 0\).
\end{itemize}
Analogous choices can also be made for \(V\).
\end{Remark}
Since $V$ is symmetric, \eqref{eq:kernel_tail_bounds} implies
\begin{equation}\label{eq:Vloss_bound}
\sup_{v\in(0,L)}\int_0^v V(v-v^\ast,v)\,dv^\ast<\infty.
\end{equation}
In fact, since $V$ is symmetric, we have
\[
V(v-v^\ast,v)=V(v,v-v^\ast).
\]
Therefore, by the change of variable $s=v-v^\ast$,
\[
\int_0^v V(v-v^\ast,v)\,dv^\ast
=
\int_0^v V(v,s)\,ds.
\]
Because $0<s<v<L$, it follows that $e^{-s}\ge e^{-L}$, and hence
\[
\int_0^v V(v,s)\,ds
\le
e^{L}\int_0^v V(v,s)e^{-s}\,ds
\le
e^{L}\int_0^\infty V(v,s)e^{-s}\,ds.
\]
Taking the supremum over $v\in(0,L)$ and using \eqref{eq:kernel_tail_bounds}, we obtain \eqref{eq:Vloss_bound}.

These conditions are compatible with polynomial-type kernels because of the exponential tail \eqref{eq:mode_extension}.

\begin{Lemma}[Lipschitz continuity of $Q_{\mathrm{coag},m}$ on $B$]\label{lem:Lip_Qcoag_m}
Suppose that all conditions in Assumption \ref{ass:kernel_conditions} hold. Then for each $m=0,\dots,N$ there exists a constant $C>0$, depending only on
$M$, $\{\Psi_n\}_{n=0}^N$, $T$, $K$, and $L$, such that for all ${\bf f},{\bf g}\in B$,
\begin{multline}\label{eq:Lip_Qcoag_m}
\int_0^L e^{2\lambda r(v)^{-\beta}}
\big|Q_{\mathrm{coag},m}({\bf f})(v)-Q_{\mathrm{coag},m}({\bf g})(v)\big|^2\,dv \\
\le
C\int_0^L e^{2\lambda r(v)^{-\beta}}|{\bf f}(v)-{\bf g}(v)|^2\,dv
+
C\,e^{2\lambda r(L)^{-\beta}}|{\bf f}(L)-{\bf g}(L)|^2.
\end{multline}
\end{Lemma}

\begin{proof}
We define
\[
\mathcal F(v,t):=\sum_{n=0}^N f_n(v)\Psi_n(t),\qquad
\mathcal G(v,t):=\sum_{n=0}^N g_n(v)\Psi_n(t),\qquad
\mathcal H:=\mathcal F-\mathcal G.
\]
Since $N$ is fixed and $\Psi_0,\dots,\Psi_N\in C([0,T])$, there exists a constant $C>0$,
depending only on $\{\Psi_n\}_{n=0}^N$ and $T$, such that
\begin{equation}\label{eq:FG_uniform_bound}
|\mathcal F(v,t)|+|\mathcal G(v,t)|\le C M,\qquad (v,t)\in(0,L)\times(0,T),
\end{equation}
for all ${\bf f},{\bf g}\in B$. Moreover, by extension \eqref{eq:mode_extension},
\begin{equation}\label{eq:FG_tail_bound}
|\mathcal F(v,t)|+|\mathcal G(v,t)|\le C M e^{-(v-L)},\qquad v>L,\ t\in(0,T).
\end{equation}
Using the orthonormality of $\{\Psi_n\}_{n=0}^N$ in $L^2_{e^{-2t}}(0,T)$, we also have
\begin{equation}\label{eq:H_coeff_identity}
\int_0^T e^{-2t}|\mathcal H(v,t)|^2\,dt
=
\begin{cases}
|{\bf f}(v)-{\bf g}(v)|^2, & 0<v<L,\\[1mm]
e^{-2(v-L)}|{\bf f}(L)-{\bf g}(L)|^2, & v\ge L.
\end{cases}
\end{equation}

For each $m=0,\dots,N$, write
\[
Q_{\mathrm{coag},m}({\bf f})(v)-Q_{\mathrm{coag},m}({\bf g})(v)
= I_m^{(1)}(v)-I_m^{(2)}(v),
\]
where
\begin{align*}
I_m^{(1)}(v)
&:=
\frac12\int_0^T\int_0^v
K(v-v^\ast,v^\ast)
\Big(\mathcal F(v-v^\ast,t)\mathcal F(v^\ast,t)-\mathcal G(v-v^\ast,t)\mathcal G(v^\ast,t)\Big)
e^{-2t}\Psi_m(t)\,dv^\ast dt,\\
I_m^{(2)}(v)
&:=
\int_0^T
\Big(
\mathcal F(v,t)\int_0^\infty K(v,v^\ast)\mathcal F(v^\ast,t)\,dv^\ast
-
\mathcal G(v,t)\int_0^\infty K(v,v^\ast)\mathcal G(v^\ast,t)\,dv^\ast
\Big)
e^{-2t}\Psi_m(t)\,dt.
\end{align*}
Hence,
\begin{equation}\label{eq:split_gain_loss}
\big|Q_{\mathrm{coag},m}({\bf f})(v)-Q_{\mathrm{coag},m}({\bf g})(v)\big|^2
\le 2|I_m^{(1)}(v)|^2+2|I_m^{(2)}(v)|^2.
\end{equation}

\medskip
\noindent{\bf Step 1: Estimate of the first term.}
Using
\[
\mathcal F(v-v^\ast,t)\mathcal F(v^\ast,t)-\mathcal G(v-v^\ast,t)\mathcal G(v^\ast,t)
=
\mathcal H(v-v^\ast,t)\mathcal F(v^\ast,t)+\mathcal G(v-v^\ast,t)\mathcal H(v^\ast,t),
\]
Cauchy--Schwarz in $t$, and $\|\Psi_m\|_{L^2_{e^{-2t}}(0,T)}=1$, we obtain
\begin{align}
|I_m^{(1)}(v)|^2
&\le
C\int_0^T e^{-2t}
\Bigg|
\int_0^v K(v-v^\ast,v^\ast)
\Big(
\mathcal H(v-v^\ast,t)\mathcal F(v^\ast,t)+\mathcal G(v-v^\ast,t)\mathcal H(v^\ast,t)
\Big)\,dv^\ast
\Bigg|^2dt
\nonumber\\
&\le
C\int_0^T e^{-2t}
\Bigg[
\Big(\int_0^v K(v-v^\ast,v^\ast)|\mathcal H(v-v^\ast,t)|\,dv^\ast\Big)^2
\nonumber \\
&\hspace{6cm}
+
\Big(\int_0^v K(v-v^\ast,v^\ast)|\mathcal H(v^\ast,t)|\,dv^\ast\Big)^2
\Bigg]dt,
\label{eq:gain_pre_est}
\end{align}
where we used \eqref{eq:FG_uniform_bound}. Since $v\in(0,L)$, both variables in this integral stay in $(0,L)$, and
\eqref{eq:kernel_bounds} implies
\[
\sup_{x\in(0,L)}\int_0^L K(x,y)^2\,dy<\infty.
\]
Therefore, by the Cauchy--Schwarz inequality with respect to the variable $v^\ast$,
\[
\Big(\int_0^v K(v-v^\ast,v^\ast)|\mathcal H(v-v^\ast,t)|\,dv^\ast\Big)^2
\le
C\int_0^v |\mathcal H(s,t)|^2\,ds,
\]
and similarly,
\[
\Big(\int_0^v K(v-v^\ast,v^\ast)| \mathcal H(v^\ast,t)|\,dv^\ast\Big)^2
\le
C\int_0^v |\mathcal H(s,t)|^2\,ds.
\]
Here we denote $C>0$ as a constant 
depending only on $\{\Psi_n\}_{n=0}^N$ and $T$, that varies from line to line.
Substituting these bounds into \eqref{eq:gain_pre_est} and then using \eqref{eq:H_coeff_identity}, we get
\begin{equation}\label{eq:gain_pointwise_final}
|I_m^{(1)}(v)|^2
\le
C\int_0^v |{\bf f}(s)-{\bf g}(s)|^2\,ds,
\qquad v\in(0,L).
\end{equation}

Since $r(v)=v-v_0$ is increasing, the Carleman weight $e^{2\lambda r(v)^{-\beta}}$ is decreasing on $[0,L]$. Hence, by Fubini's theorem and \eqref{eq:gain_pointwise_final},
\begin{align}
\int_0^L e^{2\lambda r(v)^{-\beta}}|I_m^{(1)}(v)|^2\,dv
&\le
C\int_0^L e^{2\lambda r(v)^{-\beta}}\int_0^v |{\bf f}(s)-{\bf g}(s)|^2\,ds\,dv
\nonumber\\
&=
C\int_0^L\Big(\int_s^L e^{2\lambda r(v)^{-\beta}}\,dv\Big)|{\bf f}(s)-{\bf g}(s)|^2\,ds
\nonumber\\
&\le
CL\int_0^L e^{2\lambda r(s)^{-\beta}}|{\bf f}(s)-{\bf g}(s)|^2\,ds.
\label{eq:gain_weighted_final}
\end{align}

\medskip
\noindent{\bf Step 2: Estimate of the second term.}
We write
\[
I_m^{(2)}(v)=J_{m,1}(v)+J_{m,2}(v),
\]
where
\begin{align*}
J_{m,1}(v)
&:=
\int_0^T
\mathcal H(v,t)\Big(\int_0^\infty K(v,v^\ast)\mathcal F(v^\ast,t)\,dv^\ast\Big)e^{-2t}\Psi_m(t)\,dt,\\
J_{m,2}(v)
&:=
\int_0^T
\mathcal G(v,t)\Big(\int_0^\infty K(v,v^\ast)\mathcal H(v^\ast,t)\,dv^\ast\Big)e^{-2t}\Psi_m(t)\,dt.
\end{align*}

We first bound the inner factor in $J_{m,1}$. To this end, we split the $v^\ast$-integral into $(0,L)$ and $(L,\infty)$. On $(0,L)$, using
\eqref{eq:FG_uniform_bound} and \eqref{eq:kernel_bounds},
\[
\int_0^L K(v,v^\ast)|\mathcal F(v^\ast,t)|\,dv^\ast
\le
CM\Big(\int_0^L K(v,v^\ast)^2\,dv^\ast\Big)^{1/2}
\le C.
\]
Here, again, we denote $C>0$ as a constant 
depending only on $\{\Psi_n\}_{n=0}^N$ and $T$, that varies from line to line.
On $(L,\infty)$, \eqref{eq:FG_tail_bound} and the additional tail bound on $K$ yield
\[
\int_L^\infty K(v,v^\ast)|\mathcal F(v^\ast,t)|\,dv^\ast
\le
CM\int_L^\infty K(v,v^\ast)e^{-(v^\ast-L)}\,dv^\ast
\le C.
\]
Hence,
\begin{equation}\label{eq:loss_inner_F_bound}
\sup_{(v,t)\in(0,L)\times(0,T)}
\Big|\int_0^\infty K(v,v^\ast)\mathcal F(v^\ast,t)\,dv^\ast\Big|
\le C.
\end{equation}
Using \eqref{eq:loss_inner_F_bound}, the Cauchy--Schwarz inequality with respect to $t$, and \eqref{eq:H_coeff_identity}, we obtain
\begin{equation*}
|J_{m,1}(v)|^2
\le
C\int_0^T e^{-2t}|\mathcal H(v,t)|^2\,dt
=
C|{\bf f}(v)-{\bf g}(v)|^2,
\qquad v\in(0,L).
\end{equation*}
Therefore,
\begin{equation}\label{eq:Jm1_weighted}
\int_0^L e^{2\lambda r(v)^{-\beta}}|J_{m,1}(v)|^2\,dv
\le
C\int_0^L e^{2\lambda r(v)^{-\beta}}|{\bf f}(v)-{\bf g}(v)|^2\,dv.
\end{equation}

We next treat $J_{m,2}$. Again, we split the $v^\ast$-integral into $(0,L)$ and $(L,\infty)$. By \eqref{eq:kernel_bounds} and the Cauchy--Schwarz inequality,
\[
\int_0^L K(v,v^\ast)|\mathcal H(v^\ast,t)|\,dv^\ast
\le
C\Big(\int_0^L |\mathcal H(v^\ast,t)|^2\,dv^\ast\Big)^{1/2}.
\]
Using the tail extension and the additional tail bound on $K$, we get
\[
\int_L^\infty K(v,v^\ast)|\mathcal H(v^\ast,t)|\,dv^\ast
\le
C|{\bf f}(L)-{\bf g}(L)|.
\]
Thus,
\[
\Big|\int_0^\infty K(v,v^\ast)\mathcal H(v^\ast,t)\,dv^\ast\Big|
\le
C\Big(\|\mathcal H(\cdot,t)\|_{L^2(0,L)}+|{\bf f}(L)-{\bf g}(L)|\Big).
\]
Combining this with \eqref{eq:FG_uniform_bound} and the Cauchy--Schwarz inequality in $t$, we get
\begin{align}
|J_{m,2}(v)|^2
&\le
C\int_0^T e^{-2t}
\Big(\|\mathcal H(\cdot,t)\|_{L^2(0,L)}^2+|{\bf f}(L)-{\bf g}(L)|^2\Big)\,dt
\nonumber\\
&\le
C\int_0^L |{\bf f}(s)-{\bf g}(s)|^2\,ds
+
C|{\bf f}(L)-{\bf g}(L)|^2.
\label{eq:Jm2_pointwise}
\end{align}
Since $e^{2\lambda r(v)^{-\beta}}$ is decreasing on $[0,L]$, \eqref{eq:Jm2_pointwise} implies
\begin{align}
\int_0^L e^{2\lambda r(v)^{-\beta}}|J_{m,2}(v)|^2\,dv
&\le
C\int_0^L e^{2\lambda r(v)^{-\beta}}\,dv\int_0^L |{\bf f}(s)-{\bf g}(s)|^2\,ds
\nonumber \\
&\hspace{4cm}+
C\int_0^L e^{2\lambda r(v)^{-\beta}}\,dv\,|{\bf f}(L)-{\bf g}(L)|^2
\nonumber\\
&\le
C\int_0^L e^{2\lambda r(s)^{-\beta}}|{\bf f}(s)-{\bf g}(s)|^2\,ds
+
C\,e^{2\lambda r(L)^{-\beta}}|{\bf f}(L)-{\bf g}(L)|^2.
\label{eq:Jm2_weighted}
\end{align}

Here, the constant $C>0$ depends on $\{\Psi_n\}_{n=0}^N$ and $T$ and also on $L$.

Combining \eqref{eq:Jm1_weighted} and \eqref{eq:Jm2_weighted}, we obtain
\begin{equation}\label{eq:loss_weighted_final}
\int_0^L e^{2\lambda r(v)^{-\beta}}|I_m^{(2)}(v)|^2\,dv
\le
C\int_0^L e^{2\lambda r(v)^{-\beta}}|{\bf f}(v)-{\bf g}(v)|^2\,dv
+
C\,e^{2\lambda r(L)^{-\beta}}|{\bf f}(L)-{\bf g}(L)|^2.
\end{equation}

\medskip
\noindent{\bf Step 3: Conclusion.}
Finally, \eqref{eq:split_gain_loss}, \eqref{eq:gain_weighted_final}, and \eqref{eq:loss_weighted_final} yield \eqref{eq:Lip_Qcoag_m}.
\end{proof}

\begin{Lemma}[Lipschitz continuity of $Q_{\mathrm{frag},m}$ on $B$]\label{lem:Lip_Qfrag_m}
Suppose that all conditions in Assumption \ref{ass:kernel_conditions} hold. Then for each $m=0,\dots,N$ there exists a constant $C>0$, depending only on
$\{\Psi_n\}_{n=0}^N$, $T$, $V$, and $L$ (and the extension \eqref{eq:mode_extension}), such that for all ${\bf f},{\bf g}\in B$,
\begin{multline}
\label{eq:Lip_Qfrag_m}
\int_0^L e^{2\lambda r(v)^{-\beta}}
\big|Q_{\mathrm{frag},m}({\bf f})(v)-Q_{\mathrm{frag},m}({\bf g})(v)\big|^2\,dv \\
\le
C\int_0^L e^{2\lambda r(v)^{-\beta}}|{\bf f}(v)-{\bf g}(v)|^2\,dv
+
C\,e^{2\lambda r(L)^{-\beta}}|{\bf f}(L)-{\bf g}(L)|^2.
\end{multline}
\end{Lemma}

\begin{proof}
Let $\mathcal  F$, $\mathcal  G$, and $\mathcal  H$ be as in the proof of Lemma~\ref{lem:Lip_Qcoag_m}. In particular,
\[
\int_0^T e^{-2t}|\mathcal  H(v,t)|^2\,dt
=
\begin{cases}
|{\bf f}(v)-{\bf g}(v)|^2, & v\in(0,L),\\[1mm]
e^{-2(v-L)}|{\bf f}(L)-{\bf g}(L)|^2, & v>L,
\end{cases}
\]
where for $v>L$ we use the extension \eqref{eq:mode_extension}.

Since $Q_{\mathrm{frag}}$ is linear, we have
\[
Q_{\mathrm{frag},m}({\bf f})(v)-Q_{\mathrm{frag},m}({\bf g})(v)
=
\int_0^T Q_{\mathrm{frag}}(\mathcal H)(v,t)e^{-2t}\Psi_m(t)\,dt .
\]
By \eqref{frag},
\[
Q_{\mathrm{frag}}(\mathcal  H)(v,t)
=
-\mathcal  H(v,t)\int_0^v V(v-v^\ast,v)\,dv^\ast
+2\int_0^\infty V(v,v^\ast)\mathcal  H(v+v^\ast,t)\,dv^\ast .
\]
Therefore,
\[
Q_{\mathrm{frag},m}({\bf f})(v)-Q_{\mathrm{frag},m}({\bf g})(v)
=
J_m^{(1)}(v)+J_m^{(2)}(v),
\]
where
\begin{align*}
J_m^{(1)}(v)
&:=
-\int_0^T
\mathcal  H(v,t)\Big(\int_0^v V(v-v^\ast,v)\,dv^\ast\Big)e^{-2t}\Psi_m(t)\,dt,\\
J_m^{(2)}(v)
&:=
2\int_0^T
\Big(\int_0^\infty V(v,v^\ast)\mathcal H(v+v^\ast,t)\,dv^\ast\Big)e^{-2t}\Psi_m(t)\,dt .
\end{align*}
Hence,
\begin{equation}\label{eq:frag_split}
\big|Q_{\mathrm{frag},m}({\bf f})(v)-Q_{\mathrm{frag},m}({\bf g})(v)\big|^2
\le 2|J_m^{(1)}(v)|^2+2|J_m^{(2)}(v)|^2 .
\end{equation}

For the first term, by the Cauchy--Schwarz inequality in $t$, $\|\Psi_m\|_{L^2_{e^{-2t}}(0,T)}=1$, and \eqref{eq:Vloss_bound},
\[
|J_m^{(1)}(v)|^2
\le
\Big(\int_0^v V(v-v^\ast,v)\,dv^\ast\Big)^2
\int_0^T e^{-2t}|\mathcal  H(v,t)|^2\,dt
\le
C|{\bf f}(v)-{\bf g}(v)|^2 ,
\]
for all $v\in(0,L)$. Thus,
\begin{equation}\label{eq:frag_loss_est}
\int_0^L e^{2\lambda r(v)^{-\beta}}|J_m^{(1)}(v)|^2\,dv
\le
C\int_0^L e^{2\lambda r(v)^{-\beta}}|{\bf f}(v)-{\bf g}(v)|^2\,dv .
\end{equation}

We next estimate the second term. For fixed $v\in(0,L)$, split
\[
\int_0^\infty V(v,v^\ast)\mathcal  H(v+v^\ast,t)\,dv^\ast
=
\int_0^{L-v} V(v,v^\ast)\mathcal  H(v+v^\ast,t)\,dv^\ast
+
\int_{L-v}^\infty V(v,v^\ast)\mathcal  H(v+v^\ast,t)\,dv^\ast .
\]

For the first part, since $v+v^\ast\in(0,L)$ when $0<v^\ast<L-v$, the Cauchy--Schwarz inequality and \eqref{eq:kernel_bounds} give
\[
\Big|\int_0^{L-v} V(v,v^\ast)\mathcal  H(v+v^\ast,t)\,dv^\ast\Big|^2
\le
C\int_v^L |\mathcal  H(s,t)|^2\,ds .
\]
Here, again, we denote $C>0$ as a constant 
depending only on $\{\Psi_n\}_{n=0}^N$,$V$, $L$ and $T$, that varies from line to line.
For the second part, using the extension \eqref{eq:mode_extension}, we arrive at
\[
\mathcal  H(v+v^\ast,t)=\mathcal  H(L,t)e^{-(v+v^\ast-L)},\qquad v^\ast>L-v,
\]
and therefore
\begin{align*}
\Big|\int_{L-v}^\infty V(v,v^\ast)\mathcal  H(v+v^\ast,t)\,dv^\ast\Big|
&\le
|\mathcal  H(L,t)|\int_{L-v}^\infty V(v,v^\ast)e^{-(v+v^\ast-L)}\,dv^\ast \\
&\le
C|\mathcal  H(L,t)|,
\end{align*}
where in the last step we used \eqref{eq:kernel_tail_bounds}.

Combining the two parts, we find
\[
\Big|\int_0^\infty V(v,v^\ast)\mathcal  H(v+v^\ast,t)\,dv^\ast\Big|^2
\le
C\int_v^L |\mathcal  H(s,t)|^2\,ds + C|\mathcal  H(L,t)|^2 .
\]
Applying the Cauchy--Schwarz inequality in  $t$ again yields
\[
|J_m^{(2)}(v)|^2
\le
C\int_0^T e^{-2t}\int_v^L |\mathcal  H(s,t)|^2\,ds\,dt
+
C\int_0^T e^{-2t}|\mathcal  H(L,t)|^2\,dt .
\]
Therefore,
\[
|J_m^{(2)}(v)|^2
\le
C\int_v^L |{\bf f}(s)-{\bf g}(s)|^2\,ds
+
C|{\bf f}(L)-{\bf g}(L)|^2 .
\]
Multiplying both sides of the above inequality by $e^{2\lambda r(v)^{-\beta}}$, integrating over $(0,L)$, and using Fubini's theorem together with the monotonicity of $e^{2\lambda r(v)^{-\beta}}$, we obtain
\begin{align}
\int_0^L e^{2\lambda r(v)^{-\beta}}|J_m^{(2)}(v)|^2\,dv
&\le
C\int_0^L e^{2\lambda r(v)^{-\beta}}\int_v^L |{\bf f}(s)-{\bf g}(s)|^2\,ds\,dv
+
C\int_0^L e^{2\lambda r(v)^{-\beta}}\,dv\,|{\bf f}(L)-{\bf g}(L)|^2
\nonumber\\
&\le
C\int_0^L e^{2\lambda r(v)^{-\beta}}|{\bf f}(v)-{\bf g}(v)|^2\,dv
+
C\,e^{2\lambda r(L)^{-\beta}}|{\bf f}(L)-{\bf g}(L)|^2 .
\label{eq:frag_gain_est}
\end{align}

Finally, \eqref{eq:frag_split}, \eqref{eq:frag_loss_est}, and \eqref{eq:frag_gain_est} imply \eqref{eq:Lip_Qfrag_m}.
\end{proof}

\begin{Lemma}[Lipschitz continuity of $Q_m$ on $B$]\label{lem:Lip_Q_m}
Assume Assumption~\ref{ass:kernel_conditions}. Then, for each $m=0,\dots,N$, there exists a constant $C>0$, depending only on
$M$, $\{\Psi_n\}_{n=0}^N$, $T$, $K$, $V$, and $L$, such that for all ${\bf f},{\bf g}\in B$,
\begin{multline}\label{eq:Lip_Q_m}
\int_0^L e^{2\lambda r(v)^{-\beta}}
\big|Q_m({\bf f})(v)-Q_m({\bf g})(v)\big|^2\,dv \\
\le
C\int_0^L e^{2\lambda r(v)^{-\beta}}|{\bf f}(v)-{\bf g}(v)|^2\,dv
+
C\,e^{2\lambda r(L)^{-\beta}}|{\bf f}(L)-{\bf g}(L)|^2.
\end{multline}
\end{Lemma}

\begin{proof}
By Definition~\ref{def:projected_Q},
\[
Q_m({\bf f})(v)-Q_m({\bf g})(v)
=
\Big(Q_{\mathrm{coag},m}({\bf f})(v)-Q_{\mathrm{coag},m}({\bf g})(v)\Big)
+
\Big(Q_{\mathrm{frag},m}({\bf f})(v)-Q_{\mathrm{frag},m}({\bf g})(v)\Big).
\]
Hence, by the elementary inequality $|a+b|^2\le 2|a|^2+2|b|^2$,
\begin{align*}
\big|Q_m({\bf f})(v)-Q_m({\bf g})(v)\big|^2
&\le
2\big|Q_{\mathrm{coag},m}({\bf f})(v)-Q_{\mathrm{coag},m}({\bf g})(v)\big|^2 \\
&\quad
+2\big|Q_{\mathrm{frag},m}({\bf f})(v)-Q_{\mathrm{frag},m}({\bf g})(v)\big|^2 .
\end{align*}
Multiplying by $e^{2\lambda r(v)^{-\beta}}$ and integrating over $(0,L)$, we obtain
\begin{align*}
&\int_0^L e^{2\lambda r(v)^{-\beta}}
\big|Q_m({\bf f})(v)-Q_m({\bf g})(v)\big|^2\,dv \\
&\le
2\int_0^L e^{2\lambda r(v)^{-\beta}}
\big|Q_{\mathrm{coag},m}({\bf f})(v)-Q_{\mathrm{coag},m}({\bf g})(v)\big|^2\,dv \\
&\quad
+2\int_0^L e^{2\lambda r(v)^{-\beta}}
\big|Q_{\mathrm{frag},m}({\bf f})(v)-Q_{\mathrm{frag},m}({\bf g})(v)\big|^2\,dv .
\end{align*}
Applying Lemmas~\ref{lem:Lip_Qcoag_m} and \ref{lem:Lip_Qfrag_m}, and enlarging the constant if necessary, we obtain \eqref{eq:Lip_Q_m}.
\end{proof}

\section{Time-dimension reduction}\label{sec:time_reduction}

In this section, we eliminate the time variable by expanding the solution $f(v,t)$ in a truncated
Legendre--exponential basis with respect to $t$. This procedure transforms the original time-dependent inverse
problem into a coupled time-independent system for the expansion coefficients of $f$.

Let $\{\Psi_n(t)\}_{n=0}^{\infty}$ be the Legendre--exponential basis on $[0,T]$ defined as in
Section~\ref{sec:legendre_review}. We expand $f$ as
\begin{equation}\label{expansion}
	f(v,t) = \sum_{n=0}^{\infty} f_n(v)\Psi_n(t),
	\qquad (v,t)\in (0,\infty)\times (0,T),
\end{equation}
where
\[
f_n(v) = \int_0^T e^{-2t}f(v,t)\Psi_n(t)\,dt, \qquad v\in(0,\infty).
\]

Substituting \eqref{expansion} into the governing equation \eqref{coagfrag}, we obtain
\begin{equation}\label{expansion:E1}
f_t(v,t)
=
-b(v)\sum_{n=0}^{\infty} f_n'(v)\Psi_n(t)
+\sum_{n=0}^{\infty} f_n''(v)\Psi_n(t)
+Q\!\left(\sum_{n=0}^{\infty} f_n(v)\Psi_n(t)\right),
\end{equation}
where $Q$ is the coagulation--fragmentation operator defined in \eqref{Collision}, \eqref{coag}, and \eqref{frag}.

Under the regularity assumption of Proposition~\ref{prop:first_derivative_basis}, we may differentiate the
expansion \eqref{expansion} term-by-term in $t$. Therefore, \eqref{expansion:E1} can be rewritten as
\begin{equation}\label{expansion:E1_termwise}
\sum_{n=0}^{\infty} f_n(v)\Psi_n'(t)
=
-b(v)\sum_{n=0}^{\infty} f_n'(v)\Psi_n(t)
+\sum_{n=0}^{\infty} f_n''(v)\Psi_n(t)
+Q\!\left(\sum_{n=0}^{\infty} f_n(v)\Psi_n(t)\right),
\end{equation}
for $(v,t)\in(0,L)\times(0,T)$, where the series converge in $L^2_{e^{-2t}}\big((0,T);L^2(0,L)\big)$.

Accordingly, \eqref{expansion:E1_termwise} is approximated by
\begin{equation}\label{expansion:E1_truncated_N}
\sum_{n=0}^{N} f_n(v)\Psi_n'(t)
=
-b(v)\sum_{n=0}^{N} f_n'(v)\Psi_n(t)
+\sum_{n=0}^{N} f_n''(v)\Psi_n(t)
+Q\!\left(\sum_{n=0}^{N} f_n(v)\Psi_n(t)\right),
\end{equation}
for $(v,t)\in(0,\infty)\times(0, T)$, where $N$ is a cutoff number chosen later in the numerical study.

For each $m\in\{0,1,\dots,N\}$, multiply both sides of \eqref{expansion:E1_truncated_N} by $e^{-2t}\Psi_m(t)$
and integrate over $(0,T)$. Using the orthonormality
\[
\int_0^T e^{-2t}\Psi_n(t)\Psi_m(t)\,dt=\delta_{mn},
\]
we obtain
\begin{equation}\label{eq:projected_m1}
-f_m''(v) + b(v) f_m'(v) + \sum_{n=0}^{N} s_{mn}\, f_n(v)
=
\int_0^T Q\!\left(\sum_{n=0}^{N} f_n(v)\Psi_n(t)\right)
\, e^{-2t}\Psi_m(t)\, dt,
\end{equation}
for $v\in(0,\infty)$,
where
\[
s_{mn}:=\int_0^T e^{-2t}\Psi_n'(t)\Psi_m(t)\,dt .
\]

To solve the inverse problem stated in Problem~\ref{prob:inverse_initial_density}, we restrict \eqref{eq:projected_m1} to
the computational domain $(0,L)$ and use Definition~\ref{def:projected_Q} to write
\begin{equation}\label{eq:projected_m}
-f_m''(v) + b(v) f_m'(v) + \sum_{n=0}^{N} s_{mn}\, f_n(v)
=
Q_m({\bf f}^N)(v),
\qquad v\in(0,L).
\end{equation}

\begin{Remark}\label{rem7}
Equation \eqref{eq:projected_m} is a central ingredient of our numerical method. Its derivation relies on truncating the Fourier expansion of $f$ with respect to the Legendre--exponential basis $\{\Psi_n\}_{n\ge 0}$. A natural question is why this particular basis is chosen among the many orthonormal bases of $L^2(0,T)$. The reasons are as follows.
\begin{enumerate}
    \item In the derivation of \eqref{eq:projected_m1}, we need Proposition~\ref{prop:first_derivative_basis} in order to represent the time derivative in the form
    \[
    f_t(v,t)=\sum_{n=0}^\infty f_n(v)\Psi_n'(t),
    \]
    as in \eqref{expansion:E1_termwise}. This property may fail for other choices of orthonormal bases.

    \item Another important requirement is that no basis function should have identically vanishing derivative. Indeed, if there exists an index $n_0$ such that
    \[
    \Psi_{n_0}'(t)=0 \qquad \text{for all } t\in(0,T),
    \]
    Then the corresponding Fourier mode $f_{n_0}$ would disappear from the left-hand side of \eqref{expansion:E1_termwise}, thereby reducing the accuracy of the numerical method. This issue occurs, for example, for the classical Legendre polynomial basis and the standard trigonometric Fourier basis, whose first basis element is constant. In such cases, the mode $f_0(v)$ does not contribute to the equation for $f_t$. By contrast, the Legendre--exponential basis avoids this difficulty; see Remark~\ref{rem2}. We also refer the reader to \cite[Figures 3 and 4]{VanLeNguyen}, where the reconstruction of the initial data for the compressible anisotropic Navier--Stokes equation is compared with and without the exponential weight in the basis. Those numerical results show that the Legendre--exponential basis leads to significantly better reconstructions, whereas the classical Legendre basis without the exponential weight does not provide satisfactory solutions to the inverse problem.
\end{enumerate}
\end{Remark}

\medskip
\noindent\textbf{Boundary conditions for $f_m$.}
Recall from \eqref{obs} that the time-dependent boundary observations are
\[
\phi_0(t)=f(0,t),\quad \phi_L(t)=f(L,t),\quad
\psi_0(t)=\partial_v f(0,t),\quad \psi_L(t)=\partial_v f(L,t),
\qquad t\in[0,T].
\]
For each $m=0,\dots,N$, define the corresponding Legendre--exponential coefficients
\[
\phi_{0,m}:=\int_0^T e^{-2t}\phi_0(t)\Psi_m(t)\,dt,\qquad
\phi_{L,m}:=\int_0^T e^{-2t}\phi_L(t)\Psi_m(t)\,dt,
\]
\[
\psi_{0,m}:=\int_0^T e^{-2t}\psi_0(t)\Psi_m(t)\,dt,\qquad
\psi_{L,m}:=\int_0^T e^{-2t}\psi_L(t)\Psi_m(t)\,dt.
\]
Then the coefficient functions $\{f_m\}_{m=0}^N$ satisfy
\begin{equation}\label{eq:bc_modes}
f_m(0)=\phi_{0,m},\qquad f_m(L)=\phi_{L,m},\qquad
f_m'(0)=\psi_{0,m},\qquad f_m'(L)=\psi_{L,m},\qquad m=0,\dots,N.
\end{equation}
In particular, since $f(0,t)=0$ in \eqref{coagfrag}, we have $\phi_{0,m}=0$ and thus $f_m(0)=0$ for all $m=0,\dots,N$.

\medskip
\noindent\textbf{The reduced system.}
Combining \eqref{eq:projected_m} and \eqref{eq:bc_modes}, we obtain a coupled system of ODEs for
\[
{\bf f}^N(v)=\begin{bmatrix}
f_0(v) & f_1(v) & \cdots & f_N(v)
\end{bmatrix}^{\top},
\]
namely, for each $m=0,1,\dots,N$,
\begin{equation}\label{eq:ODE_system_modes}
\begin{cases}
-f_m''(v) + b(v) f_m'(v) + \displaystyle\sum_{n=0}^{N} s_{mn}\, f_n(v)
= Q_m({\bf f}^N)(v),
& v \in (0, L),\\[2mm]
f_m(0) = 0,\quad f_m(L) = \phi_{L,m},\quad f_m'(0)=\psi_{0,m},\quad f_m'(L)=\psi_{L,m}.
\end{cases}
\end{equation}

\begin{Remark}[Extension beyond the computational domain]\label{rem:extension_Qfrag}
After restricting \eqref{eq:ODE_system_modes} to $v\in(0,L)$, the right-hand side still involves integrals over $(0,\infty)$
through the fragmentation gain term in $Q_{\mathrm{frag}}$, namely
\[
2\int_0^\infty V(v,v^\ast)\,f(v+v^\ast,t)\,dv^\ast,
\]
which requires values of $f(\cdot,t)$ (and hence the modes $\{f_n\}_{n=0}^N$) at sizes $v+v^\ast>L$ even when $v\in(0,L)$.
Therefore, to make $Q_m({\bf f}^N)(v)$ well defined on $(0,L)$, one must prescribe an extension of the coefficient functions
$\{f_n\}_{n=0}^N$ from $(0,L)$ to $(0,\infty)$.

In this work, we adopt a continuous exponential tail extension: for each $n=0,\dots, N$, we set
\[
f_n(v):=f_n(L)\,e^{-(v-L)},\qquad v>L.
\]
All occurrences of $Q_m({\bf f}^N)(v)$ in \eqref{eq:ODE_system_modes} are understood with this extension when evaluating the
fragmentation and coagulation operators.
\end{Remark}
\begin{Remark}
   Let us note that, in order to solve the classical Boltzmann equation numerically, one also imposes a domain truncation (see \cite{filbet2006solving,gamba2009spectral}). However, instead of using a continuous tail extension as above, the solution is assumed to be periodic on the truncated domain.
\end{Remark}
\begin{Remark}\label{rem6}
Equation~\eqref{eq:ODE_system_modes} constitutes the \textbf{time-dimensional reduction model}. Solving
\eqref{eq:ODE_system_modes} is the main step toward addressing Problem~\ref{prob:inverse_initial_density}, since it yields
${\bf f}^N(v)=\big(f_0(v),\ldots,f_N(v)\big)^\top$. Once ${\bf f}^N$ is obtained, the solution $f(v,t)$ can be reconstructed via the
truncated expansion \eqref{expansion}, i.e.,
\[
f(v,t)\approx \sum_{n=0}^{N} f_n(v)\Psi_n(t), \qquad (v,t)\in(0,L)\times(0,T).
\]
Consequently, the initial density is recovered by evaluating the reconstructed solution at $t=0$,
\[
f^0(v)=f(v,0)\approx \sum_{n=0}^{N} f_n(v)\Psi_n(0), \qquad v\in[0,L].
\]
\end{Remark}

Solving \eqref{eq:ODE_system_modes} is nontrivial due to the nonlinear term $Q_m({\bf f}^N)(v)$ and the intricate
coagulation--fragmentation structure of the operator $Q$; see \eqref{Collision}--\eqref{frag} and Definition~\ref{def:projected_Q}.
To control this nonlinearity in the analysis below, we work under the following admissibility condition.

\section{The Carleman--Picard iteration}\label{sec:carleman_picard}

As discussed in Remark~\ref{rem6}, computing the solution of \eqref{eq:ODE_system_modes} is the central step in addressing
Problem~\ref{prob:inverse_initial_density}. A natural baseline strategy is to recover ${\bf f}^{*}$ by solving
\eqref{eq:ODE_system_modes} in a least-squares sense, i.e., by minimizing the nonlinear functional
\begin{multline}\label{eq:plain_LS_functional}
J_{\rm ncvx}(\bm{\varphi})
:=\sum_{m=0}^{N}\Bigg[
\int_{0}^{L}\Big|
-\varphi_m''(v)+b(v)\varphi_m'(v)+\sum_{n=0}^{N}s_{mn}\varphi_n(v)
- Q_m(\bm{\varphi})(v)
\Big|^{2}\,dv
+|\varphi_m(0)|^{2}
\\
+|\varphi_m(L)-\phi_{L,m}|^{2}
+|\varphi_m'(0)-\psi_{0,m}|^{2}
+|\varphi_m'(L)-\psi_{L,m}|^{2}
+\epsilon\|\varphi_m\|_{H^{3}(0,L)}^{2}
\Bigg],
\qquad \bm{\varphi}\in B.
\end{multline}
Here, the admissible set $B$ is introduced in Definition~\ref{def:admissible_set}, and $\epsilon>0$ is a small regularization
parameter. This formulation is appealing and, in principle, robust, and therefore widely used in the mathematical and engineering
communities. However, because ${\bf Q}=\begin{bmatrix}Q_0 & \cdots & Q_N\end{bmatrix}^\top$ has a complicated coagulation--fragmentation
structure, the functional $J_{\rm ncvx}$ is generally nonconvex and may possess multiple local minimizers. Consequently, a direct
minimization of \eqref{eq:plain_LS_functional} may fail to recover ${\bf f}^{*}$ unless a sufficiently accurate initial guess is
available.

To overcome this difficulty, we combine a Picard-type linearization with a Carleman-weighted least-squares minimization. The resulting
Carleman--Picard scheme exploits the one-dimensional Carleman estimate in Lemma~\ref{thm:carleman_1d_pointwise} to enforce stability
and to guarantee convergence to ${\bf f}^{*}$.

\medskip
\noindent\textbf{Carleman--Picard update.}
Let
\[
{\bf f}^{N,(0)}(v)=\begin{bmatrix}
f_0^{(0)}(v) & f_1^{(0)}(v) & \cdots & f_N^{(0)}(v)
\end{bmatrix}^{\top}\in B
\]
be an initial guess, not necessarily close to the exact solution ${\bf f}^{*}$. Assume that, for some $k\ge 0$, the iterate
\[
{\bf f}^{N,(k)}(v)=\begin{bmatrix}
f_0^{(k)}(v) & f_1^{(k)}(v) & \cdots & f_N^{(k)}(v)
\end{bmatrix}^{\top}\in B
\]
is known. The next iterate ${\bf f}^{N,(k+1)}$ is defined as the unique minimizer of the Carleman-weighted functional
\begin{equation}\label{eq:carleman_argmin}
{\bf f}^{N,(k+1)}
=
\operatorname*{arg\,min}_{\bm{\varphi}\in B}
J_{\lambda,\epsilon}^{{\bf f}^{N,(k)}}(\bm{\varphi}),
\end{equation}
where $J_{\lambda,\epsilon}^{{\bf f}^{N,(k)}}$ is given by
\begin{multline}\label{eq:carleman_functional}
J_{\lambda,\epsilon}^{{\bf f}^{N,(k)}}(\bm{\varphi})
=
\sum_{m=0}^N \Bigg[
\int_0^L e^{2\lambda r(v)^{-\beta}}
\Big|
-\varphi_m''(v) + b(v)\varphi_m'(v) + \sum_{n=0}^{N} s_{mn}\,\varphi_n(v)
- Q_m({\bf f}^{N,(k)})(v)
\Big|^2\,dv
\\
+ \lambda^4 e^{2\lambda r(0)^{-\beta}}|\varphi_m(0)|^2
+ \lambda^4 e^{2\lambda r(L)^{-\beta}}|\varphi_m(L)-\phi_{L,m}|^2
+ \lambda^4 e^{2\lambda r(0)^{-\beta}}|\varphi_m'(0)-\psi_{0,m}|^2
\\
+ \lambda^4 e^{2\lambda r(L)^{-\beta}}|\varphi_m'(L)-\psi_{L,m}|^2
+ \epsilon \|\varphi_m\|^2_{H^3(0,L)}
\Bigg].
\end{multline}
Here $\epsilon>0$ is a regularization parameter, and the Carleman weight $e^{2\lambda r(v)^{-\beta}}$ is chosen as in
Subsection~\ref{sec:carleman_1d}. Recall that, in the evaluation of $J_{\lambda,\epsilon}^{{\bf f}^{N,(k)}}$, the vector ${\bf f}^{N,(k)}$ is extended to $v>L$ as described in Remark~\ref{rem:extension_Qfrag}.

\begin{Remark}[Well-posedness of the minimization step]
Fix $k\ge 0$. Since the nonlinear term in \eqref{eq:carleman_functional} is evaluated at the known iterate ${\bf f}^{N,(k)}$,
the functional $\bm{\varphi}\mapsto J_{\lambda,\epsilon}^{{\bf f}^{N,(k)}}(\bm{\varphi})$ is a quadratic functional of $\bm{\varphi}$.
Moreover, the $H^3(0,L)$-regularization term with $\epsilon>0$ makes $J_{\lambda,\epsilon}^{{\bf f}^{N,(k)}}$ coercive and strictly
convex on $[H^3(0,L)]^{N+1}$. Since $B\subset [H^3(0,L)]^{N+1}$ is closed and convex, the minimization problem
\eqref{eq:carleman_argmin} admits a unique minimizer in $B$.
\end{Remark}

\medskip
\noindent
The theorem below guarantees the convergence of the Carleman--Picard method for solving \eqref{eq:ODE_system_modes}.

\begin{theorem}
\label{thm:carleman_picard_convergence}
Assume Assumption~\ref{ass:kernel_conditions}. Fix $\beta>0$, and let $\lambda_0$ be as in Lemma~\ref{thm:carleman_1d_pointwise}. Let ${\bf f}^{N,(0)}\in B$, and for each $k\ge 0$, define ${\bf f}^{N,(k+1)}$ by \eqref{eq:carleman_argmin}. Let ${\bf f}^{*}\in B$ be the exact solution to \eqref{eq:ODE_system_modes}. Then there exist $\lambda_1\ge \lambda_0$ and a constant $C>0$ such that, for all $\lambda\ge \lambda_1$ and all $\epsilon>0$,
\begin{multline}\label{61}
\int_0^L e^{2\lambda r(v)^{-\beta}}
|{\bf f}^{N,(k+1)}(v)-{\bf f}^*(v)|^2\,dv
+
e^{2\lambda r(L)^{-\beta}}
|{\bf f}^{N,(k+1)}(L)-{\bf f}^*(L)|^2
\\
\le
\frac{C}{\lambda^3}\Bigg[
\int_0^L e^{2\lambda r(v)^{-\beta}}|{\bf f}^{N,(k)}(v)-{\bf f}^*(v)|^2\,dv
+
e^{2\lambda r(L)^{-\beta}}|{\bf f}^{N,(k)}(L)-{\bf f}^*(L)|^2
\\
+ \epsilon \|{\bf f}^*\|_{[H^3(0, L)]^{N + 1}}^2\Bigg].
\end{multline}
Here, $C>0$ depends only on $L$, $\beta$, $v_0$, $\|b\|_{L^\infty(0,L)}$, $N$, $\{s_{mn}\}_{m,n=0}^N$, $M$, $T$, $\{\Psi_n\}_{n=0}^N$, $K$, and $V$, and is independent of $k$, $\lambda$, and $\epsilon$.

In particular, let $\rho:=\frac{C}{\lambda^3}$. If $\lambda$ is sufficiently large so that $\rho\in(0,1)$, then
\begin{equation}\label{eq:Ek_bias_bound}
\|{\bf f}^{N,(k)}-{\bf f}^*\|_\lambda^2
\le
\rho^k\|{\bf f}^{N,(0)}-{\bf f}^*\|_\lambda^2
+
\frac{\rho \epsilon}{1-\rho}\,\,\|{\bf f}^*\|_{[H^3(0,L)]^{N+1}}^2,
\qquad k\ge 0,
\end{equation}
where
\[
\|{\bf g}\|_{\lambda}^2
:=
\int_0^L e^{2\lambda r(v)^{-\beta}}|{\bf g}(v)|^2\,dv
+
e^{2\lambda r(L)^{-\beta}}|{\bf g}(L)|^2.
\]
Consequently, $\{{\bf f}^{N,(k)}\}_{k\ge 0}$ converges geometrically to an $\epsilon$-neighborhood of ${\bf f}^*$ in the weighted norm
$\|\cdot\|_\lambda$, and
\[
\limsup_{k\to\infty}\|{\bf f}^{N,(k)}-{\bf f}^*\|_\lambda^2
\le
\frac{\rho \epsilon}{1-\rho}\,\|{\bf f}^*\|_{[H^3(0,L)]^{N+1}}^2.
\]
\end{theorem}

\begin{proof}
Since $B$ is convex and ${\bf f}^{N,(k+1)}$ minimizes $J_{\lambda,\epsilon}^{{\bf f}^{N,(k)}}$ over $B$, the standard variational
inequality yields
\begin{equation*}
D J_{\lambda,\epsilon}^{{\bf f}^{N,(k)}}\!\left({\bf f}^{N,(k+1)}\right)
\Big[\,{\bf y}-{\bf f}^{N,(k+1)}\,\Big]\ge 0
\qquad \text{for all }{\bf y}\in B.
\end{equation*}
In particular, since ${\bf f}^{*}\in B$, choosing ${\bf y}={\bf f}^{*}$ gives
\begin{equation}\label{eq:VI_fstar}
D J_{\lambda,\epsilon}^{{\bf f}^{N,(k)}}\!\left({\bf f}^{N,(k+1)}\right)
\Big[\,{\bf f}^{*}-{\bf f}^{N,(k+1)}\,\Big]\ge 0.
\end{equation}
Let
\begin{equation}\label{eq:def_h}
{\bf h}:={\bf f}^{N,(k+1)}-{\bf f}^{*},\qquad 
h_m:=f_m^{(k+1)}-f_m^{*},\qquad m=0,\dots,N.
\end{equation}
Then ${\bf f}^{*}-{\bf f}^{N,(k+1)}=-{\bf h}$. By the linearity of the directional derivative in its direction argument,
\eqref{eq:VI_fstar} is equivalent to
\begin{equation*}
D J_{\lambda,\epsilon}^{{\bf f}^{N,(k)}}\!\left({\bf f}^{N,(k+1)}\right)\big[{\bf h}\big]\le 0.
\end{equation*}

Using the definition \eqref{eq:carleman_functional} of $J_{\lambda,\epsilon}^{{\bf f}^{N,(k)}}$, we observe that the term
$Q_m({\bf f}^{N,(k)})(v)$ is evaluated at the known iterate ${\bf f}^{N,(k)}$ and is thus independent of $\bm{\varphi}$ in the differentiation with respect to $\bm{\varphi}$.
Consequently, we obtain
\begin{multline}\label{eq:DJ_h_raw}
D J_{\lambda,\epsilon}^{{\bf f}^{N,(k)}}\!\left({\bf f}^{N,(k+1)}\right)\big[{\bf h}\big]
=
2\sum_{m=0}^N\Bigg[
\Big\langle e^{2\lambda r(v)^{-\beta}}
\Big(
-{f_m^{(k+1)}}''(v)+b(v){f_m^{(k+1)}}'(v)+\sum_{n=0}^N s_{mn}f_n^{(k+1)}(v)
\\
- Q_m({\bf f}^{N,(k)})(v)
\Big),
-h_m''(v)+b(v)h_m'(v)+\sum_{n=0}^N s_{mn}h_n(v)
\Big\rangle_{L^2(0,L)}
\\
+\lambda^4 e^{2\lambda r(0)^{-\beta}}\,\big(f_m^{(k+1)}(0)\big)h_m(0)
+\lambda^4 e^{2\lambda r(L)^{-\beta}}\,\big(f_m^{(k+1)}(L)-\phi_{L,m}\big)h_m(L)
\\
+\lambda^4 e^{2\lambda r(0)^{-\beta}}\,\big({f_m^{(k+1)}}'(0)-\psi_{0,m}\big)h_m'(0)
+\lambda^4 e^{2\lambda r(L)^{-\beta}}\,\big({f_m^{(k+1)}}'(L)-\psi_{L,m}\big)h_m'(L)
\\
+\epsilon\,\Big\langle f_m^{(k+1)},\, h_m\Big\rangle_{H^3(0,L)}
\Bigg] \le 0.
\end{multline}

Since ${\bf f}^{*}$ solves \eqref{eq:ODE_system_modes}, we have, for each $m=0,\dots,N$ and $v\in(0,L)$,
\[
-{f_m^{*}}''(v)+b(v){f_m^{*}}'(v)+\sum_{n=0}^N s_{mn}f_n^{*}(v)-Q_m({\bf f}^{*})(v)=0,
\]
and ${\bf f}^{*}$ satisfies the boundary data in \eqref{eq:ODE_system_modes}. Therefore,
\begin{multline}\label{eq:f_identity}
2\sum_{m=0}^N\Bigg[
\Big\langle e^{2\lambda r(v)^{-\beta}}
\Big(
-{f_m^{*}}''(v)+b(v){f_m^{*}}'(v)+\sum_{n=0}^N s_{mn}f_n^{*}(v)-Q_m({\bf f}^{*})(v)
\Big),\\
-h_m''(v)+b(v)h_m'(v)+\sum_{n=0}^N s_{mn}h_n(v)
\Big\rangle_{L^2(0,L)}
+\lambda^4 e^{2\lambda r(0)^{-\beta}}\,f_m^{*}(0)\,h_m(0)\\
+\lambda^4 e^{2\lambda r(L)^{-\beta}}\,\big(f_m^{*}(L)-\phi_{L,m}\big)\,h_m(L)
+\lambda^4 e^{2\lambda r(0)^{-\beta}}\,\big({f_m^{*}}'(0)-\psi_{0,m}\big)\,h_m'(0)\\
+\lambda^4 e^{2\lambda r(L)^{-\beta}}\,\big({f_m^{*}}'(L)-\psi_{L,m}\big)\,h_m'(L)
+\epsilon\,\Big\langle f_m^{*},\, h_m\Big\rangle_{H^3(0,L)}
\Bigg]
=
2\epsilon\sum_{m=0}^N\Big\langle f_m^{*},\, h_m\Big\rangle_{H^3(0,L)}.
\end{multline}

Define, for $m=0,\dots,N$,
\begin{equation*}
\Delta_m^{(k)}(v):=Q_m({\bf f}^{*})(v)-Q_m({\bf f}^{N,(k)})(v),\qquad v\in(0,L).
\end{equation*}
Subtracting \eqref{eq:f_identity} from \eqref{eq:DJ_h_raw}, using ${\bf f}^{N,(k+1)}={\bf f}^{*}+{\bf h}$, and applying Young's
inequality $XY\le \frac12 X^2+\frac12 Y^2$, we obtain

\begin{multline}\label{eq:energy_ineq_h}
\sum_{m = 0}^N\Bigg[
\int_{0}^L e^{2\lambda r(v)^{-\beta}}
\Big|-h_m''(v)+b(v)h_m'(v)+\sum_{n=0}^N s_{mn}h_n(v)\Big|^2\,dv
\\
+\lambda^4 e^{2\lambda r(0)^{-\beta}}\Big(|h_m(0)|^2+|h_m'(0)|^2\Big)
+\lambda^4 e^{2\lambda r(L)^{-\beta}}\Big(|h_m(L)|^2+|h_m'(L)|^2\Big)
+\epsilon \|h_m\|_{H^3(0, L)}^2
\Bigg]
\\
\le
C \sum_{m = 0}^N \int_{0}^L e^{2\lambda r(v)^{-\beta}} |\Delta^{(k)}_m(v)|^2\,dv
- 2\epsilon \sum_{m = 0}^N\Big\langle f_m^{*},\, h_m\Big\rangle_{H^3(0,L)}
\end{multline}
for some constant $C>0$. As above, $C$ denotes a generic constant depending only on $L$, $\beta$, $v_0$, $\|b\|_{L^\infty(0,L)}$, $N$, $\{s_{mn}\}_{m,n=0}^N$, $M$, $T$, $\{\Psi_n\}_{n=0}^N$, $K$, and $V$, and independent of $k$, $\lambda$, and $\epsilon$. We do not keep track of the value of $C$.

Applying Lemma \ref{lem:Lip_Q_m} and using \eqref{eq:energy_ineq_h}, we have
\begin{multline}\label{58}
\sum_{m = 0}^N\Bigg[
\int_{0}^L e^{2\lambda r(v)^{-\beta}}
\Big|-h_m''(v)+b(v)h_m'(v)+\sum_{n=0}^N s_{mn}h_n(v)\Big|^2\,dv
\\
+\lambda^4 e^{2\lambda r(0)^{-\beta}}\Big(|h_m(0)|^2+|h_m'(0)|^2\Big)
+\lambda^4 e^{2\lambda r(L)^{-\beta}}\Big(|h_m(L)|^2+|h_m'(L)|^2\Big)
+\epsilon \|h_m\|_{H^3(0, L)}^2
\Bigg]
\\
\le
C\int_0^L e^{2\lambda r(v)^{-\beta}}|{\bf f}^{N, (k)}(v)-{\bf f}^*(v)|^2\,dv
+
C\,e^{2\lambda r(L)^{-\beta}}|{\bf f}^{N, (k)}(L)-{\bf f}^*(L)|^2
\\
+ \epsilon \|{\bf f}^*\|_{[H^3(0, L)]^{N + 1}}^2 + \epsilon \|{\bf h}\|_{[H^3(0, L)]^{N + 1}}^2
\end{multline}
Since $b \in L^{\infty}(0,\infty)$, applying the inequality
\[
(X+Y+Z)^2 \ge \frac13 X^2 - Y^2 - Z^2
\]
to the left-hand side of \eqref{58}, we obtain
\begin{multline*}
\int_{0}^L e^{2\lambda r(v)^{-\beta}}
|{\bf h}''(v)|^2\,dv
- C \int_0^L e^{2\lambda r(v)^{-\beta}} |{\bf h}'(v)|^2\,dv
- C \int_0^L e^{2\lambda r(v)^{-\beta}} |{\bf h}(v)|^2\,dv
\\
+\lambda^4 e^{2\lambda r(0)^{-\beta}}\Big(|{\bf h}(0)|^2+|{\bf h}'(0)|^2\Big)
+\lambda^4 e^{2\lambda r(L)^{-\beta}}\Big(|{\bf h}(L)|^2+|{\bf h}'(L)|^2\Big)
\\
\le
C\int_0^L e^{2\lambda r(v)^{-\beta}}|{\bf f}^{N,(k)}(v)-{\bf f}^*(v)|^2\,dv
+
C\,e^{2\lambda r(L)^{-\beta}}|{\bf f}^{N,(k)}(L)-{\bf f}^*(L)|^2 
+ \epsilon \|{\bf f}^*\|_{[H^3(0, L)]^{N + 1}}^2.
\end{multline*}
Applying Corollary~\ref{cor:carleman_1d_integral} componentwise to $h_m$, summing over $m=0,\dots,N$, and choosing $\lambda$ sufficiently large, we obtain
\begin{multline}\label{eq:carleman_absorbed_h}
\int_0^L e^{2\lambda r(v)^{-\beta}}
\Big(\lambda^3|{\bf h}(v)|^2+\lambda|{\bf h}'(v)|^2\Big)\,dv
\\
+\lambda^4 e^{2\lambda r(0)^{-\beta}}\Big(|{\bf h}(0)|^2+|{\bf h}'(0)|^2\Big)
+\lambda^4 e^{2\lambda r(L)^{-\beta}}\Big(|{\bf h}(L)|^2+|{\bf h}'(L)|^2\Big)
\\
\le
C\int_0^L e^{2\lambda r(v)^{-\beta}}|{\bf f}^{N,(k)}(v)-{\bf f}^*(v)|^2\,dv
+
C\,e^{2\lambda r(L)^{-\beta}}|{\bf f}^{N,(k)}(L)-{\bf f}^*(L)|^2
\\
+ \epsilon \|{\bf f}^*\|_{[H^3(0, L)]^{N + 1}}^2.
\end{multline}
Dropping the nonnegative terms
\[
\lambda\int_0^L e^{2\lambda r(v)^{-\beta}}|{\bf h}'(v)|^2\,dv,\qquad
\lambda^4 e^{2\lambda r(0)^{-\beta}}\Big(|{\bf h}(0)|^2+|{\bf h}'(0)|^2\Big),\qquad
\lambda^4 e^{2\lambda r(L)^{-\beta}}|{\bf h}'(L)|^2,
\]
from the left-hand side of \eqref{eq:carleman_absorbed_h}, and recalling \eqref{eq:def_h}, we obtain \eqref{61}.


Define
\[
E_k:=\|{\bf f}^{N,(k)}-{\bf f}^*\|_\lambda^2.
\]
Then \eqref{61} can be written as
\[
E_{k+1}\le \rho\,E_k+\rho\,\epsilon\,\|{\bf f}^*\|_{[H^3(0,L)]^{N+1}}^2,
\qquad \rho:=\frac{C}{\lambda^3}.
\]
Assume $\rho\in(0,1)$. Iterating the recurrence yields
\[
E_k
\le
\rho^k E_0
+
\rho\epsilon\,\|{\bf f}^*\|_{[H^3(0,L)]^{N+1}}^2\sum_{j=0}^{k-1}\rho^j
=
\rho^k E_0
+
\frac{\rho \epsilon}{1-\rho}\,\|{\bf f}^*\|_{[H^3(0,L)]^{N+1}}^2,
\]
which is \eqref{eq:Ek_bias_bound}. This proves geometric convergence to an $\epsilon$-dependent neighborhood of ${\bf f}^*$.
\end{proof}

\begin{Remark}\label{rem12}
Theorem~\ref{thm:carleman_picard_convergence} and the estimate \eqref{eq:Ek_bias_bound} shows that the Carleman--Picard iteration converges to an $\epsilon$-neighborhood of the exact solution ${\bf f}^*$ in the weighted norm $\|\cdot\|_\lambda$. In particular, when the regularization parameter $\epsilon$ is chosen sufficiently small, the limiting point of the iteration can be regarded as a good approximation of the true solution ${\bf f}^*$. Therefore, in practical computations, one may expect the reconstructed coefficient vector ${\bf f}^{N,(k)}$ for large $k$ to provide an accurate approximation of ${\bf f}^*$, up to a small regularization error controlled by $\epsilon$.
\end{Remark}

Theorem~\ref{thm:carleman_picard_convergence} and Remark \ref{rem12} motivate the Carleman--Picard reconstruction procedure for Problem~\ref{prob:inverse_initial_density}, which is summarized in Algorithm~\ref{alg:carleman_picard}.

\begin{algorithm}[H]
\caption{Carleman--Picard reconstruction for Problem~\ref{prob:inverse_initial_density}}\label{alg:carleman_picard}
\begin{algorithmic}[1]
\State Choose the truncation index $N$, the Carleman parameters $\lambda$ and $\beta$, and the regularization parameter $\epsilon$. \label{s1}
\State Compute the data coefficients $\phi_{L,m}$, $\psi_{0,m}$, and $\psi_{L,m}$ for $m=0,\dots,N$ from \eqref{obs}. \label{s2}
\State Choose an initial guess ${\bf f}^{N,(0)}(v)=\begin{bmatrix}f_0^{(0)}(v)&\dots&f_N^{(0)}(v)\end{bmatrix}^\top\in B$
and a maximum number of iterations $K_{\max}>0$. \label{s3}
\For{$k=0,1,\dots,K_{\max}-1$}
\State Extend the Fourier modes: for each $n=0,\dots,N$, set
\[
f_n^{(k)}(v):=f_n^{(k)}(L)e^{-(v-L)},\qquad v>L.
\]
\State Compute
\[
{\bf f}^{N,(k+1)}
=
\operatorname*{arg\,min}_{\bm{\varphi}\in B}
J_{\lambda,\epsilon}^{{\bf f}^{N,(k)}}(\bm{\varphi}),
\]
where $J_{\lambda,\epsilon}^{{\bf f}^{N,(k)}}$ is defined in \eqref{eq:carleman_functional}. \label{s6}
\EndFor
\State \label{s8} Reconstruct the space--time density on $(0,L)\times(0,T)$ by
\[
f^{\mathrm{rec}}(v,t):=\sum_{n=0}^N f_n^{(K_{\max})}(v)\Psi_n(t).
\]
\State Output the reconstructed initial density
\[
f_{\mathrm{in}}^{\mathrm{rec}}(v):=f^{\mathrm{rec}}(v,0)=\sum_{n=0}^N f_n^{(K_{\max})}(v)\Psi_n(0),\qquad v\in[0,L].
\]
\end{algorithmic}
\end{algorithm}

In the next section, we describe the numerical implementation of Algorithm~\ref{alg:carleman_picard} and present illustrative
reconstruction examples.

\section{Numerical Experiments for the Carleman--Picard Method}\label{sec5}
In this section, we present numerical experiments to illustrate the performance of the Carleman--Picard method for Problem~\ref{prob:inverse_initial_density}. We first describe how synthetic boundary data are generated by numerically solving the forward coagulation--fragmentation equation. We then describe the numerical implementation of the reconstruction algorithm and present several tests on noisy data. The results demonstrate that the proposed method yields accurate and stable reconstructions for a variety of representative initial densities.

\subsection{Generation of synthetic data by solving the forward problem}

To generate synthetic boundary data for the inverse problem, we first solve the forward coagulation--fragmentation equation
numerically. In all experiments, we take $b(v)\equiv 1$ and choose the coagulation and fragmentation kernels
\[
K(v,v^\ast)=v+v^\ast,
\qquad
V(v,v^\ast)=v+v^\ast.
\]
The initial density $f_{\mathrm{in}}$ is selected from the family of test profiles described below.

\textbf{Truncation of the half-line.}
The forward problem is posed for $v\in(0,\infty)$. In computations, we truncate the size domain to a finite interval $(0, R)$ with
$R=10$. This truncation is justified by the rapid decay of the size distribution for large $v$ observed in our simulations; in
particular, we enforce the homogeneous boundary condition
\[
f(R,t)=0,\qquad t\in(0,T),
\]
as a numerical surrogate for the decay condition $f(v,t)\to 0$ as $v\to\infty$.

\textbf{Discretization.}
We solve the forward problem on $(0,R)\times(0,T)$ with $T=0.5$. The size variable is discretized on a uniform grid with \(N_v=241\) nodes,
\[
v_i=(i-1)\Delta v,\qquad i=1,\dots,N_v,
\]
with mesh size
\[
\Delta v=\frac{R}{N_v-1}.
\]
Time is discretized on a uniform grid with \(N_t=301\) nodes,
\[
t_n=(n-1)\Delta t,\qquad n=1,\dots,N_t,
\]
with a time step
\[
\Delta t=\frac{T}{N_t-1}.
\]
\textbf{Time stepping.}
To improve stability, we use a semi-implicit time-stepping method. Denoting by $f_i^n$ the approximation of $f(v_i,t_n)$, we
evaluate the collision operator explicitly at time level $t_n$, while the size-transport and size-diffusion terms are treated
implicitly at time level $t_{n+1}$. More precisely, at each time step, we solve
\begin{equation*}
\frac{f^{n+1}-f^n}{\Delta t}
=
-\partial_v f^{n+1}
+\partial_{vv}f^{n+1}
+Q(f^n),
\end{equation*}
where the derivatives in $v$ are approximated by standard finite differences on the size grid. The resulting linear system at each time step is solved in the least-squares sense to accommodate the boundary constraints on $(0,R)$.

\textbf{Discrete evaluation of the collision operator.}
The collision operator $Q=Q_{\mathrm{coag}}+Q_{\mathrm{frag}}$ is evaluated on the size grid by direct quadrature (discrete
summation). Since the forward computation is performed on the truncated interval $(0, R)$ with the decay surrogate $f(R,t)=0$, any
term requiring values beyond $R$ (e.g., $f(v+v^\ast,t)$ with $v+v^\ast>R$) is set to zero in the discrete implementation. This is
consistent with the truncation of the half-line and with the imposed boundary condition at $v=R$.

\textbf{Restriction to the inverse domain and noise.}
After computing the forward solution on the larger domain $(0,R)\times(0,T)$, we restrict it to the inverse domain
\[
(0,L)\times(0,T), \qquad L=2,
\]
by interpolation onto the reconstruction grid. We then define the exact boundary data
\[
\phi_L^{\rm true}(t)=f^{\rm true}(L,t),\qquad
\psi_0^{\rm true}(t)=\partial_v f^{\rm true}(0,t),\qquad
\psi_L^{\rm true}(t)=\partial_v f^{\rm true}(L,t),
\qquad t\in(0,T).
\]
Since the model imposes the boundary condition $f(0,t)=0$, we have
\[
\phi_0(t)=0,
\]
so this datum is kept exact and is not perturbed by noise.

To simulate measurement errors, we corrupt the remaining boundary observations by multiplicative noise:
\begin{align*}
\phi_L^\delta(t)&=\phi_L^{\rm true}(t)\bigl(1+\delta\,\xi_1(t)\bigr), \\
\psi_0^\delta(t)&=\psi_0^{\rm true}(t)\bigl(1+\delta\,\xi_2(t)\bigr), \\
\psi_L^\delta(t)&=\psi_L^{\rm true}(t)\bigl(1+\delta\,\xi_3(t)\bigr).
\end{align*}
for $t \in (0, T)$
where $\delta>0$ is the noise level and $\xi_1,\xi_2,\xi_3$ are independent random functions uniformly distributed in $[-1,1]$. 

\subsection{Numerical Implementation }
In this subsection, we present some details of the implementation of Algorithm \ref{alg:carleman_picard}.

{\bf Step~\ref{s1}: Choice of $N$, $\lambda$, $\beta$, and $\epsilon$.}



The parameters are chosen empirically. We first select one numerical example, namely \textbf{Test 1} below, as the reference test.

The truncation index \(N\) is determined by comparing the forward solution \(f(L,t)\) with its truncated Legendre--exponential expansion
\[
\sum_{n=0}^N f_n(L)\Psi_n(t)
\]
at the noise-free level (\(0\%\) noise). More precisely, we introduce the function \(\varphi:\mathbb{N}\to\mathbb{R}\), which measures the relative discrepancy between \(f(L,t), (t\in(0, T)\), and the truncation of its Fourier expansion with respect to the polynomial--exponential basis. It is defined by
\begin{equation}
\varphi(N)=
\frac{\left\|f(L,t)-\sum_{n=0}^N f_n(L)\Psi_n(t)\right\|_{L^\infty(0,T)}}
{\|f(L,t)\|_{L^\infty(0,T)}}.
\label{error_N}
\end{equation}

\begin{figure}[h!]
\centering
\subfloat[\label{figN_a}]{\includegraphics[width=.47\textwidth]{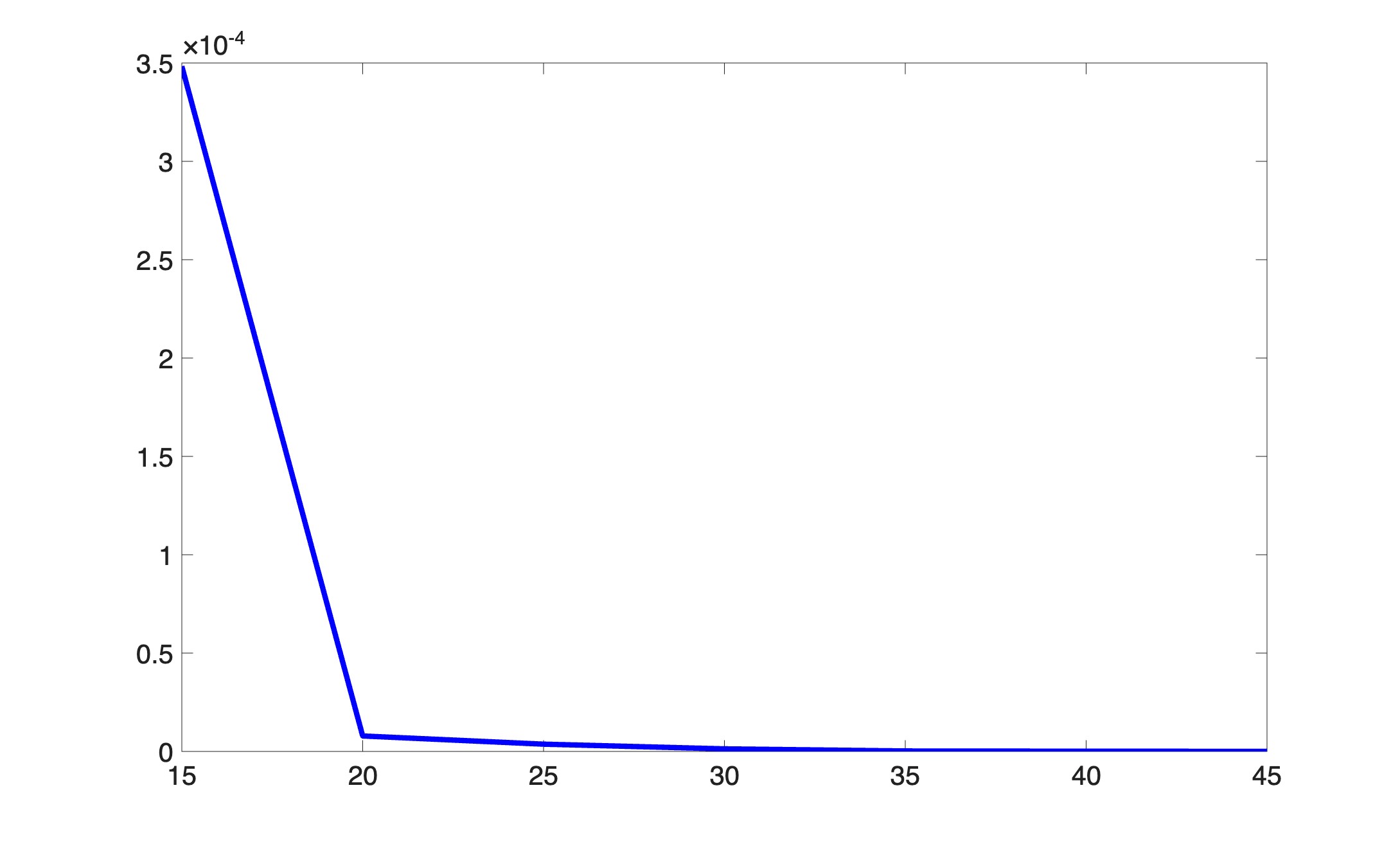}}
\quad
\subfloat[\label{figN_b}]{\includegraphics[width=.47\textwidth]{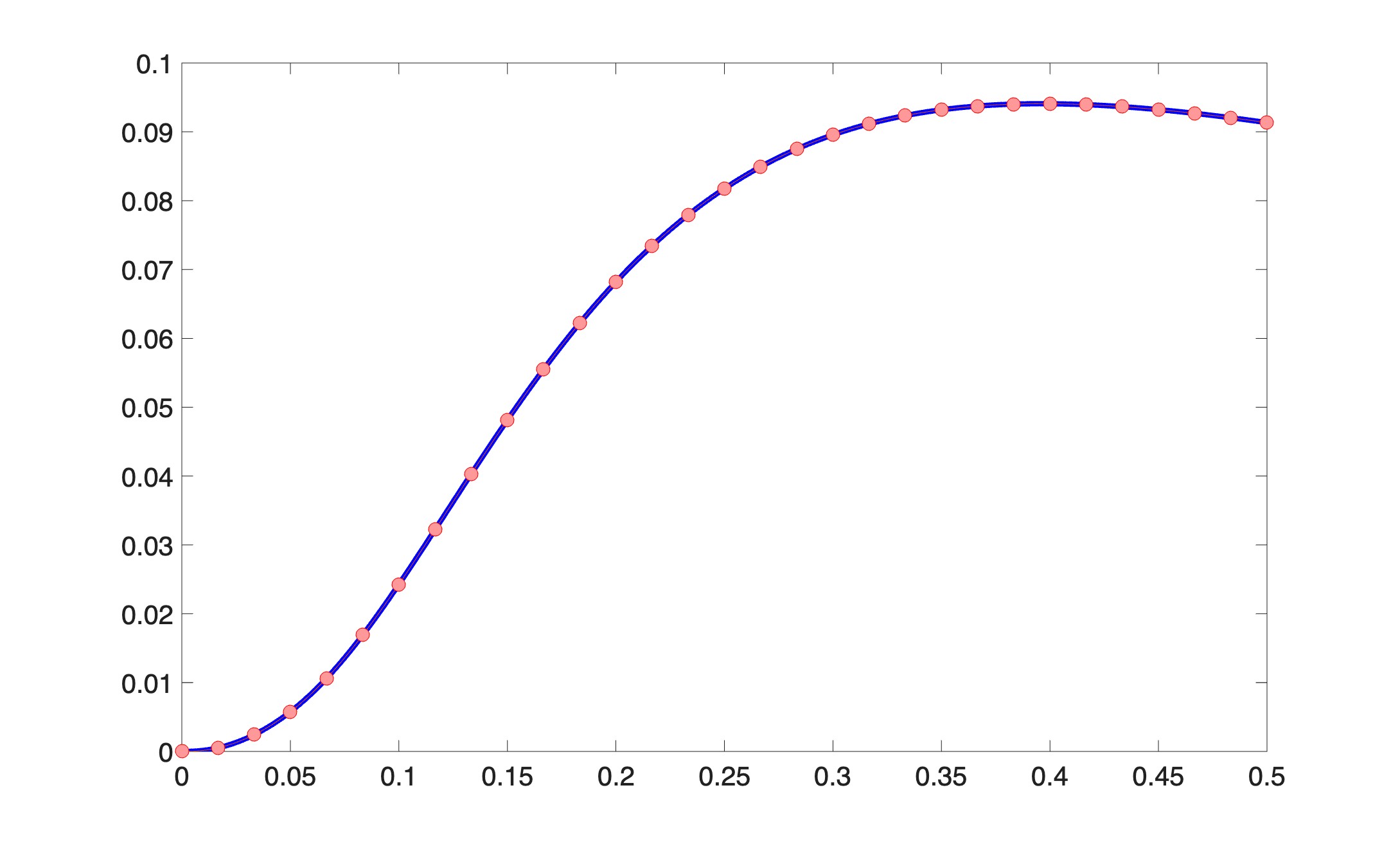}}
\caption{\label{figN}
(a) The graph of the function \(\varphi:\{15,\dots,45\}\to\mathbb{R}\) defined in \eqref{error_N}, computed from Test 1 below.
(b) Comparison between the forward solution \(f(L,t)\) and its truncated Legendre-exponential expansion with \(N=20\) over the interval \(t\in[0,T]\). The blue curve represents $f(L,t)$, while the red markers denote its truncated expansion.}
\end{figure}

It is interesting to observe from Figure ~\ref{figN_a} that the graph of \(\varphi\) has an \(L\)-curve shape. This observation somewhat confirms that $N$ serves as a regularization factor. For Test 1, we choose \(N=20\), which corresponds to the corner of the \(L\)-curve. Numerically, we also observe that larger values of \(N\), for example \(25\le N\le 45\), produce similar reconstruction results. Figure~\ref{figN_b} shows that the truncated expansion with \(N=20\) matches the forward data \(f(L,t)\) very well over the entire observation interval. The two graphs are almost indistinguishable, indicating that the first \(20\) basis functions already capture the essential temporal behavior of the signal. This observation supports the choice \(N=20\) as a suitable truncation level for the reconstruction procedure.

Once \(N\) is fixed, the remaining parameters \(\lambda\), \(\beta\), and \(\epsilon\) are selected by trial and error so that the reconstruction quality for the reference test is satisfactory. After this calibration step, the same parameter values are used for all other numerical examples reported in this paper.

In all reconstruction experiments, we take the truncation number and regularization parameters to be
\[
N=20,\qquad \beta=10,\qquad \lambda=2,\qquad \epsilon=10^{-6.5}.
\]


\textbf{Step~\ref{s3}: Initial guess and maximum number of iterations.}
For the initialization, we choose the initial guess to be the zero function, that is,
\[
{\bf f}^{N,(0)}(v)=0
\qquad \text{for all } v\in(0,L).
\]
Equivalently, each initial mode is set to
\[
f_n^{(0)}(v)=0,\qquad n=0,\dots,N.
\]
This choice provides a simple baseline initialization and is consistent with the convergence result established for the Carleman--Picard iteration. In all numerical experiments, the maximum number of iterations is fixed at
$
K_{\max}=9.
$

\textbf{Step~\ref{s6}: Numerical realization of the minimization step.}
In the implementation, the update
\[
{\bf f}^{N,(k+1)}
=
\operatorname*{arg\,min}_{\bm{\varphi}\in B}
J_{\lambda,\epsilon}^{{\bf f}^{N,(k)}}(\bm{\varphi})
\]
is carried out at the discrete level as a regularized constrained least-squares problem.

More precisely, for a given iterate ${\bf f}^{N,(k)}$, we first evaluate the coagulation--fragmentation operator at ${\bf f}^{N,(k)}$ in the physical variables $(v,t)$. The resulting function is then projected onto the truncated Legendre--exponential basis in time and rewritten in vectorized form with respect to the spatial grid and the mode index. In this way, we obtain the discrete right-hand side associated with the frozen nonlinear terms
\[
Q_m({\bf f}^{N,(k)}),\qquad m=0,\dots,N.
\]

Next, we assemble the discrete linear operator corresponding to the left-hand side of the reduced system
\[
-f_m''(v)+b(v)f_m'(v)+\sum_{n=0}^N s_{mn}f_n(v),
\qquad m=0,\dots,N,
\]
together with the Tikhonov regularization term. After freezing the nonlinear operator at the previous iterate, Step~\ref{s6} reduces to a linear constrained optimization problem. The new iterate ${\bf f}^{N,(k+1)}$ is then computed by solving the resulting regularized least-squares system subject to the boundary constraints at $v=0$ and $v=L$.
Therefore, the numerical implementation of Step~\ref{s6} consists of three main substeps: evaluation of the frozen nonlinear source at ${\bf f}^{N,(k)}$, projection onto the truncated time basis, and solution of the corresponding regularized constrained linear least-squares problem for ${\bf f}^{N,(k+1)}$. The discrete minimization problem arising at each Carleman--Picard step is solved by the MATLAB constrained least-squares solver \texttt{lsqlin}.

The other steps in Algorithm~\ref{alg:carleman_picard} are implemented straightforwardly according to their definitions. In particular, the extension beyond $v=L$, the reconstruction of the space--time density, and the recovery of the initial profile are performed directly using the corresponding formulas in the algorithm.

\subsection{Numerical reconstruction results}

In this section, we display some numerical tests.

{\bf Test 1:} For Test 1, we choose the true initial density to be
\[
f^0(v)=
\begin{cases}
\dfrac{\pi}{2}\sin(\pi v), & 0\le v\le 1,\\[1mm]
0, & v > 1.
\end{cases}
\]
The function is smooth on $(0,1)$, vanishes at both $v=0$ and $v=1$, and has compact support in $[0,1]$.
Figure~\ref{fig:test1} displays the reconstruction results for Test~1 at the noise levels $5\%$ and $10\%$.

\begin{figure}[ht]
\centering
\subfloat[Reconstruction of $f^0$ with $5\%$ noise]{
    \includegraphics[width=0.47\textwidth]{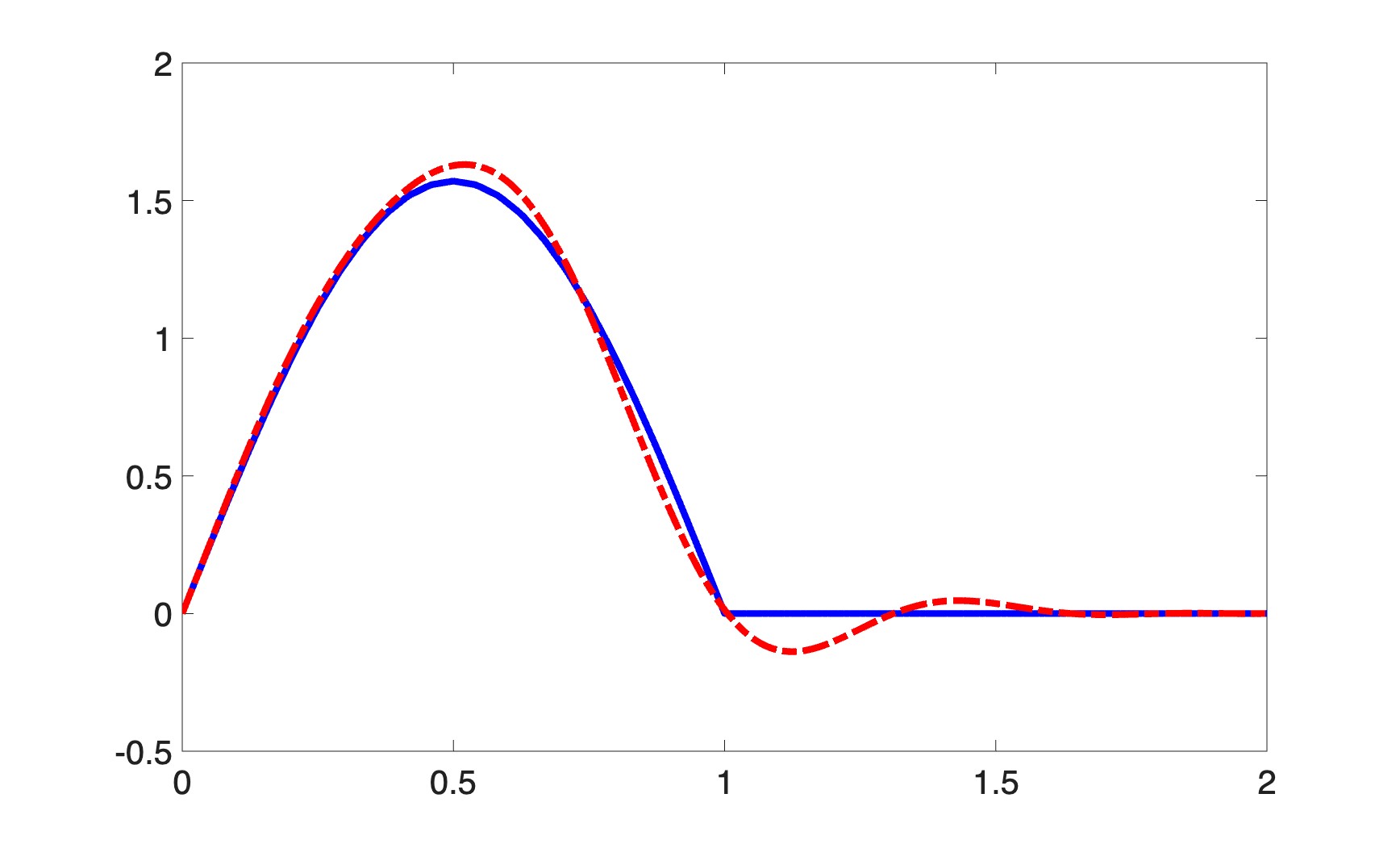}
}
\quad
\subfloat[\label{fig1b} Absolute consecutive errors with $5\%$ noise]{
    \includegraphics[width=0.47\textwidth]{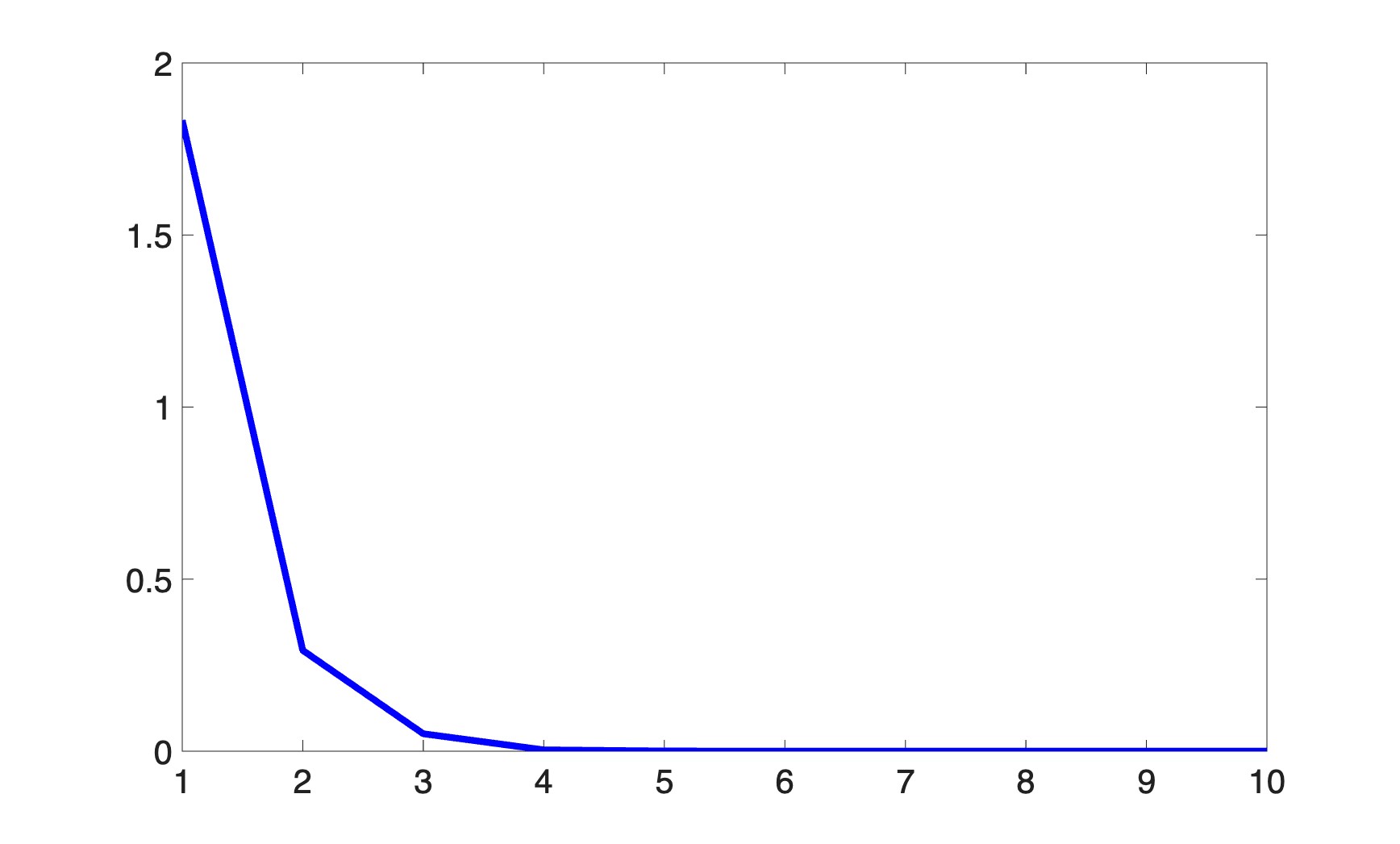}
}

\subfloat[Reconstruction of $f^0$ with $10\%$ noise]{
    \includegraphics[width=0.47\textwidth]{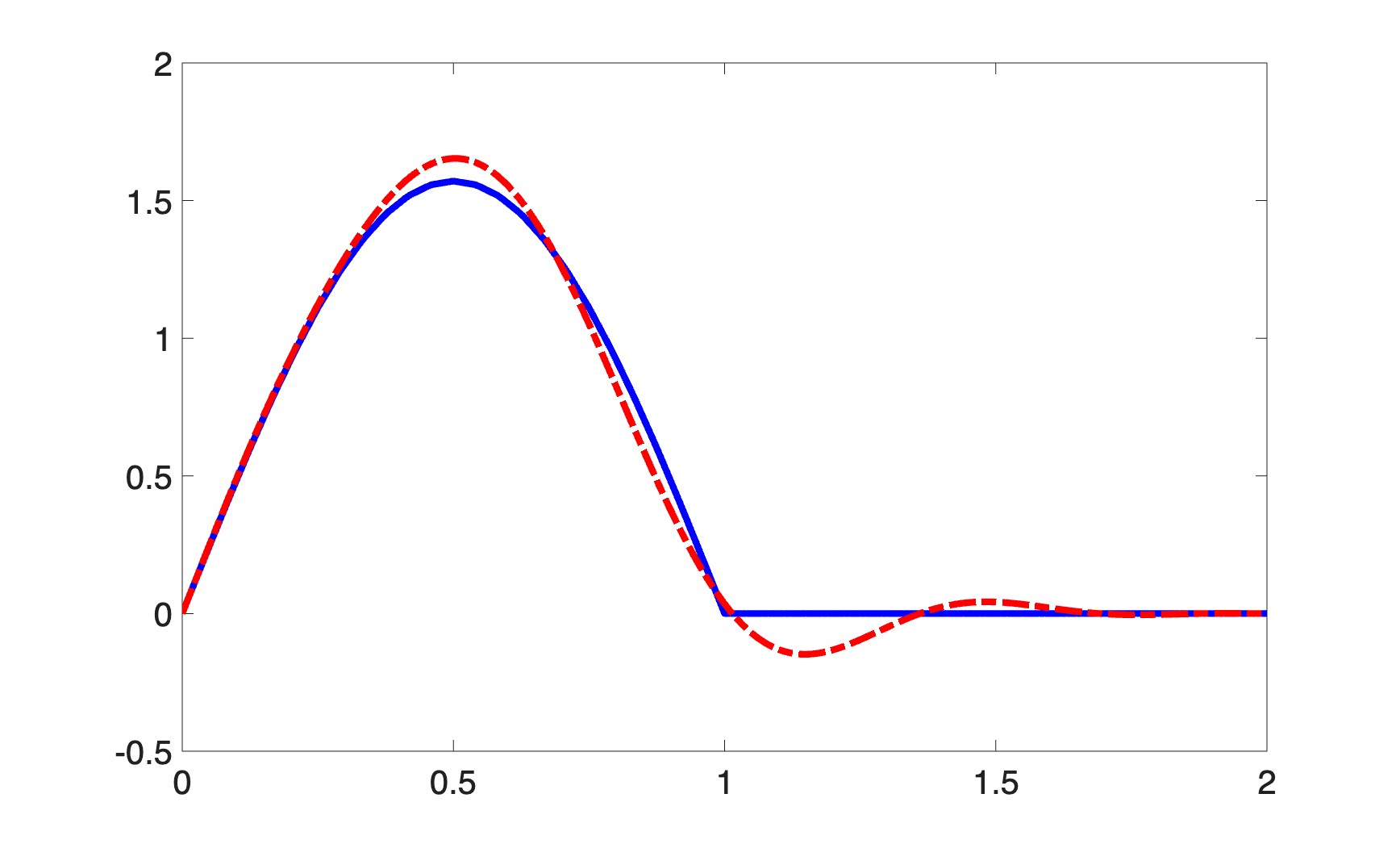}
}
\quad
\subfloat[\label{fig1d}Absolute consecutive errors with $10\%$ noise]{
    \includegraphics[width=0.47\textwidth]{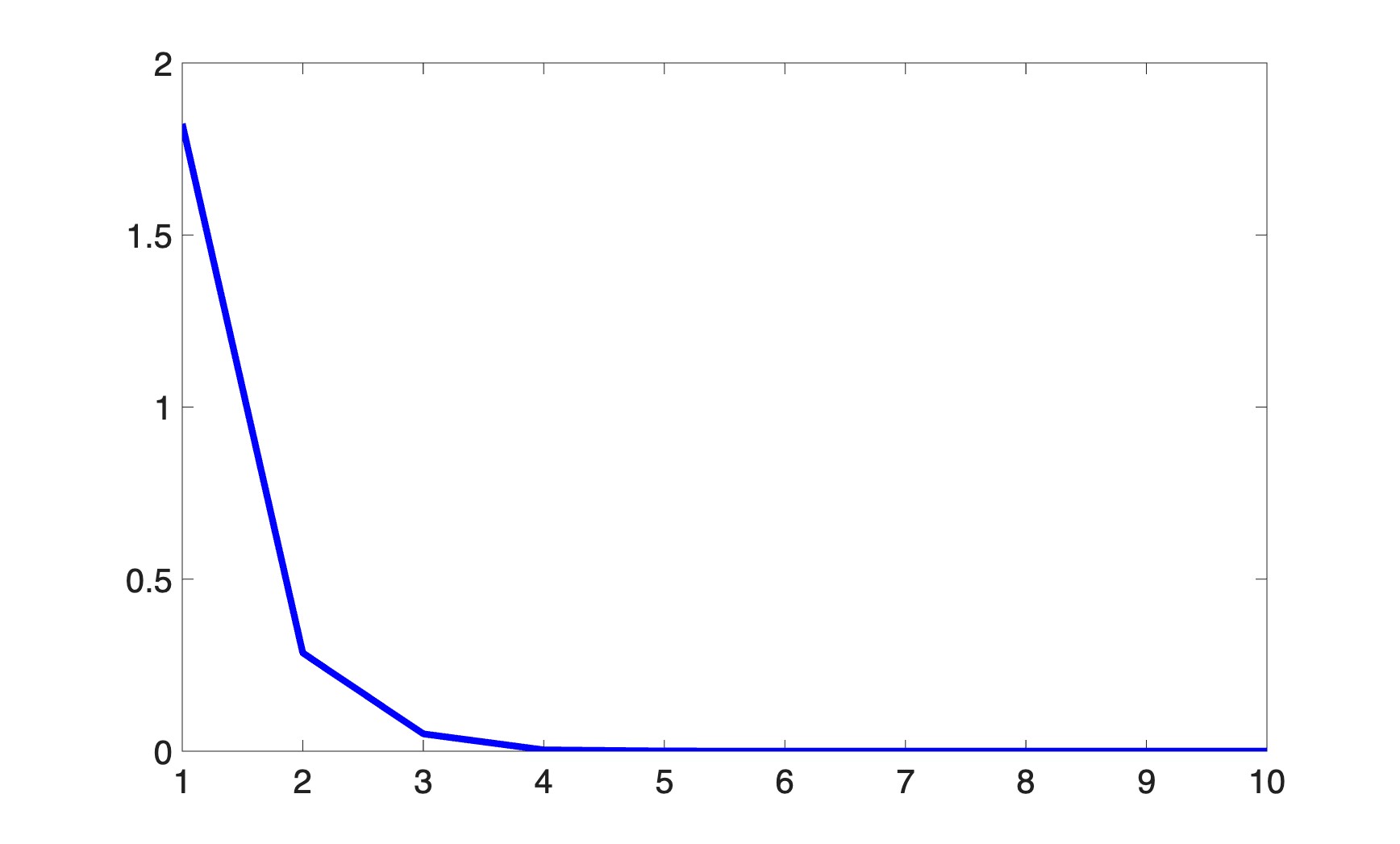}
}

\caption{Test 1. The first row shows the results for the $5\%$ noise level, while the second row shows the results for the $10\%$ noise level. In each row, the left plot compares the true initial density and its reconstruction: the solid curve denotes the true function, and the dashed curve denotes the reconstructed one. The right plot shows the absolute consecutive error $\|f^{\mathrm{rec},(k+1)}-f^{\mathrm{rec},(k)}\|_{L^\infty(0,L)}$.}
\label{fig:test1}
\end{figure}

The results in Figure~\ref{fig:test1} show that the proposed method reconstructs the true initial density well for both noise levels. In particular, the reconstructed function matches the true one closely in shape and amplitude for $\delta=5\%$, and it remains stable and accurate for $\delta=10\%$. To quantify the reconstruction accuracy, we use the relative $L^2$ error $\frac{\|f^{0,\mathrm{true}}-f^{0,\mathrm{rec}}\|_{L^2(0,L)}}{\|f^{0,\mathrm{true}}\|_{L^2(0,L)}}$ and the relative $L^\infty$ error $\frac{\|f^{0,\mathrm{true}}-f^{0,\mathrm{rec}}\|_{L^\infty(0,L)}}{\|f^{0,\mathrm{true}}\|_{L^\infty(0,L)}}$. For the $5\%$ noise level, these errors are $0.0700$ and $0.0880$, respectively. For the $10\%$ noise level, they are $0.0781$ and $0.0942$, respectively. These results indicate that the proposed method yields accurate reconstructions and remains robust in the presence of noise. Moreover, the absolute consecutive errors decrease rapidly over iterations, confirming the stable numerical behavior of the Carleman--Picard scheme.

{\bf Test 2:} For Test 2, the true initial density is given by the probability density function of the Gaussian distribution
\[
f^0(v)=\frac{1}{\sqrt{2\pi}\sigma}\exp\!\left(-\frac{(v-\mu)^2}{2\sigma^2}\right),
\qquad \mu=0.7,\quad \sigma=0.2.
\]
The function is smooth and positive on $(0, L)$, with its peak located at $v=0.7$.
Figure~\ref{fig:test2} displays the reconstruction results for Test~1 at the noise levels $5\%$ and $10\%$.

\begin{figure}[ht]
\centering
\subfloat[Reconstruction of $f^0$ with $5\%$ noise]{
    \includegraphics[width=0.47\textwidth]{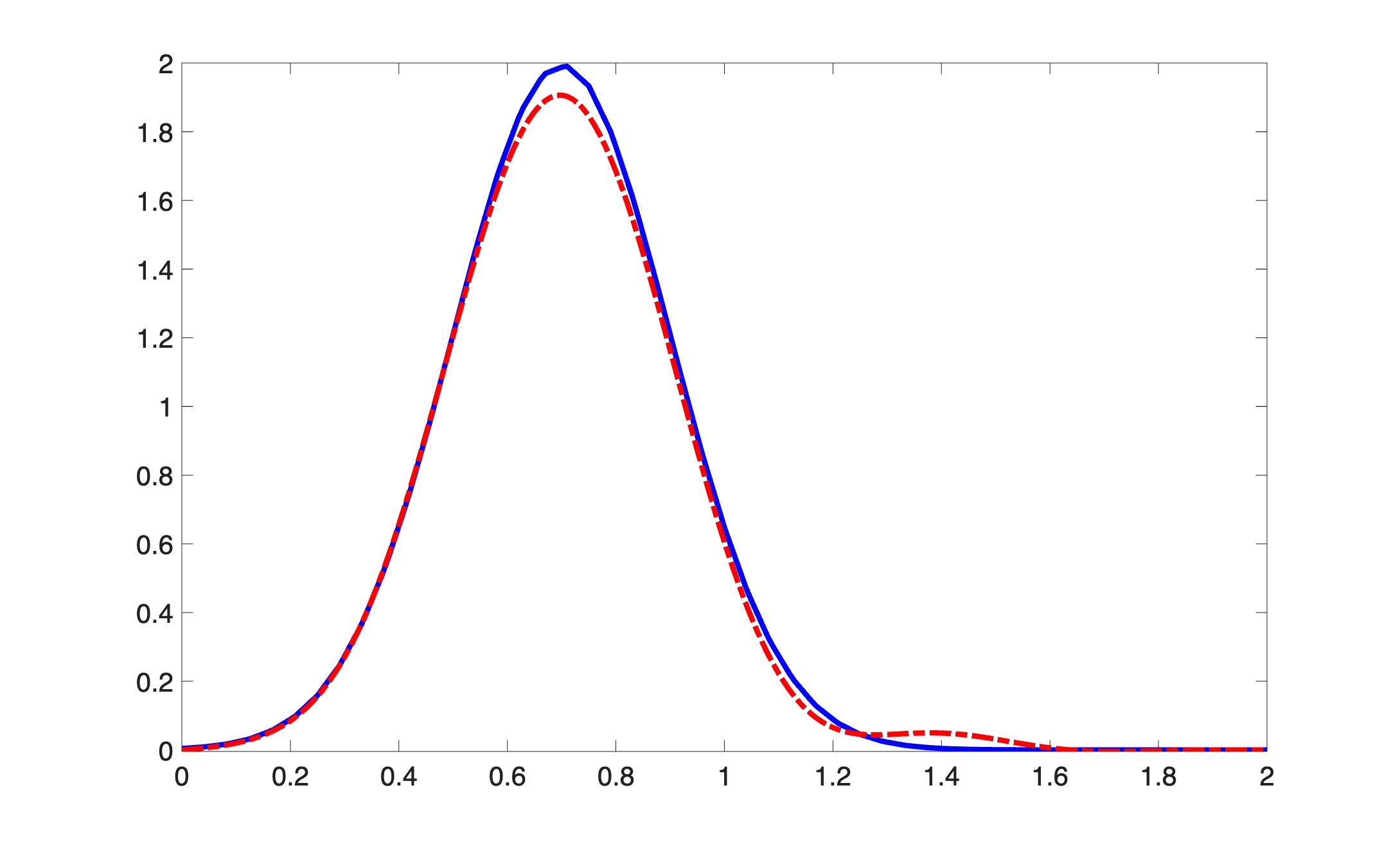}
}
\quad
\subfloat[Absolute consecutive errors with $5\%$ noise]{
    \includegraphics[width=0.47\textwidth]{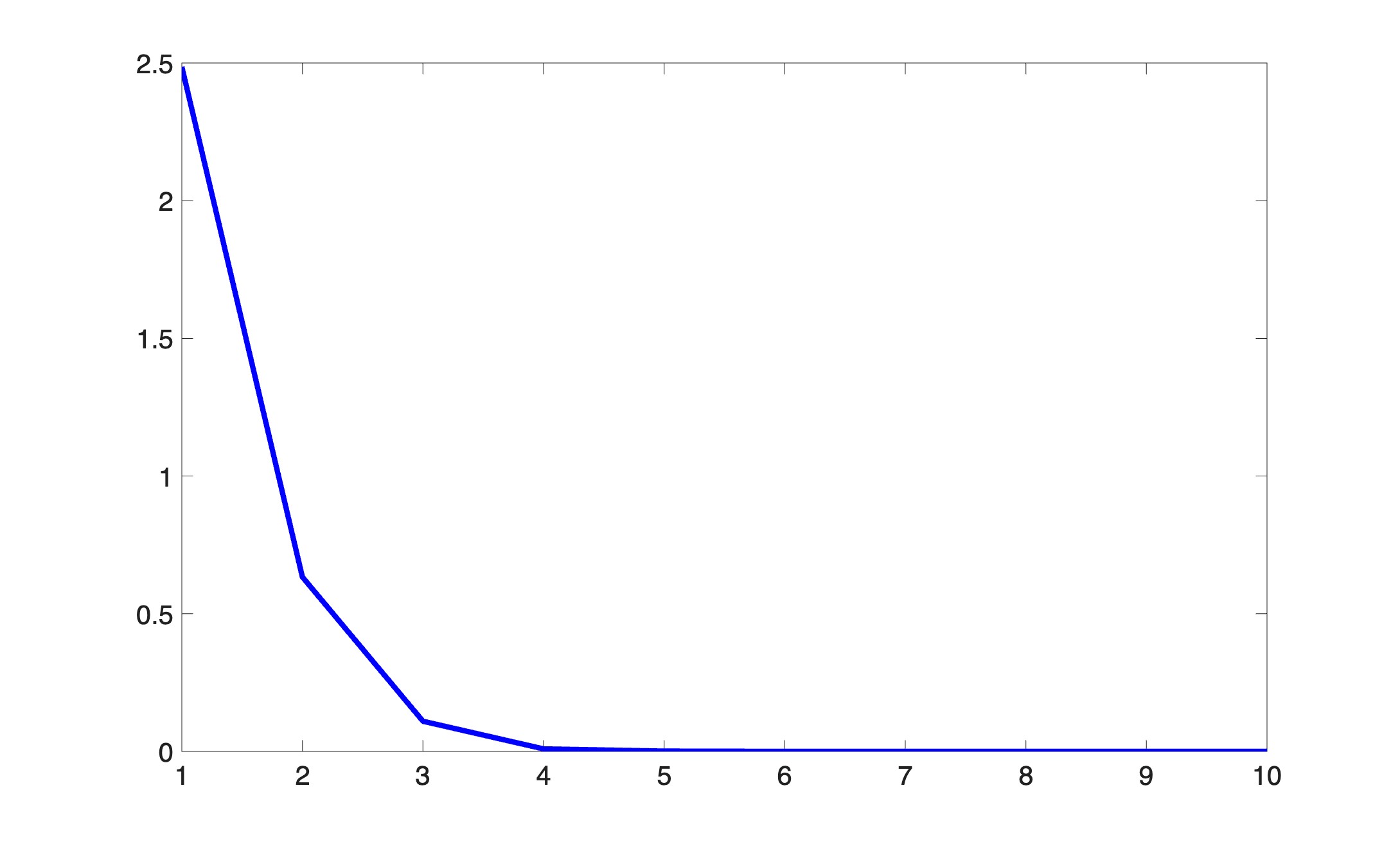}
}

\subfloat[Reconstruction of $f^0$ with $10\%$ noise]{
    \includegraphics[width=0.47\textwidth]{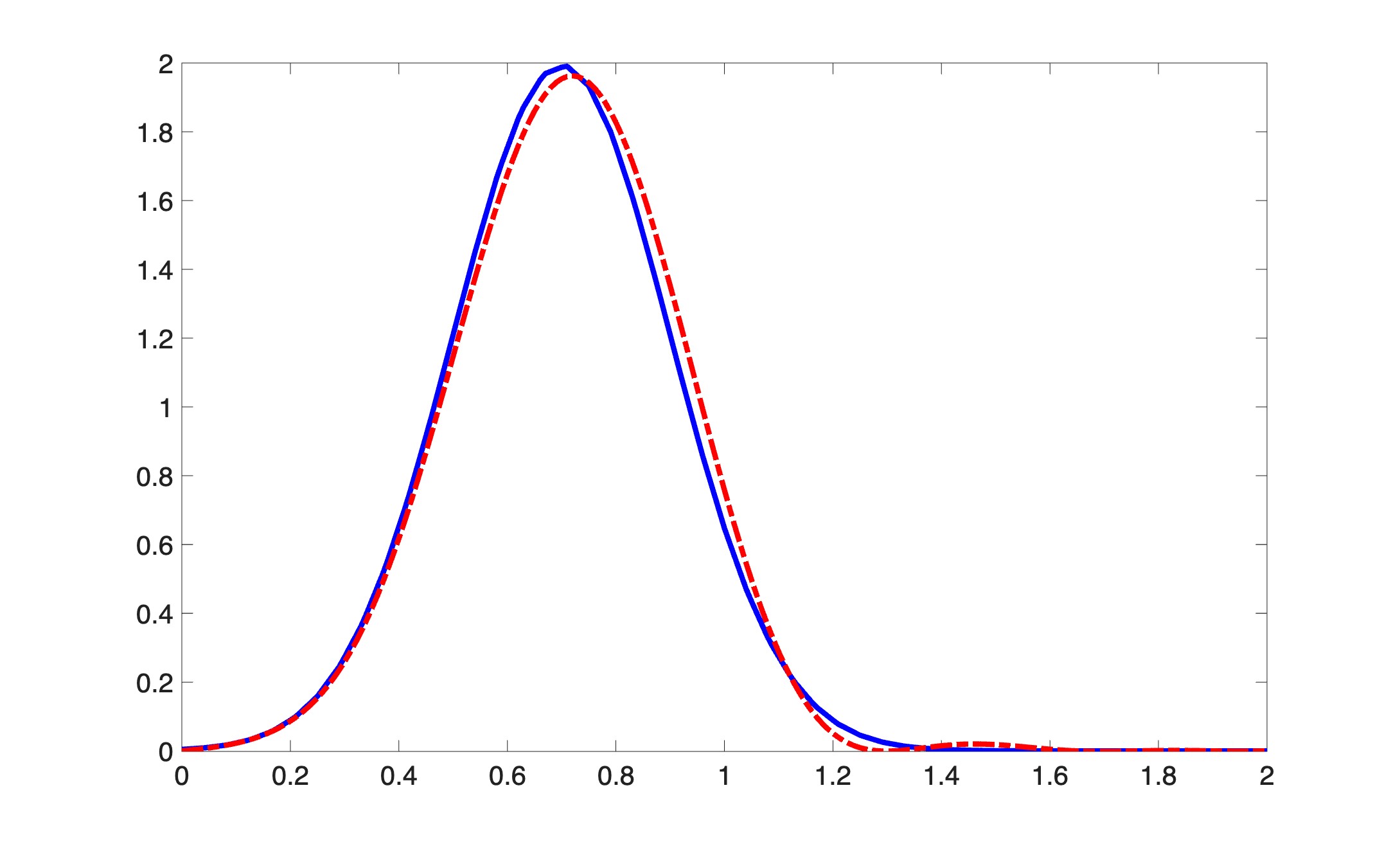}
}
\quad
\subfloat[Absolute consecutive errors with $10\%$ noise]{
    \includegraphics[width=0.47\textwidth]{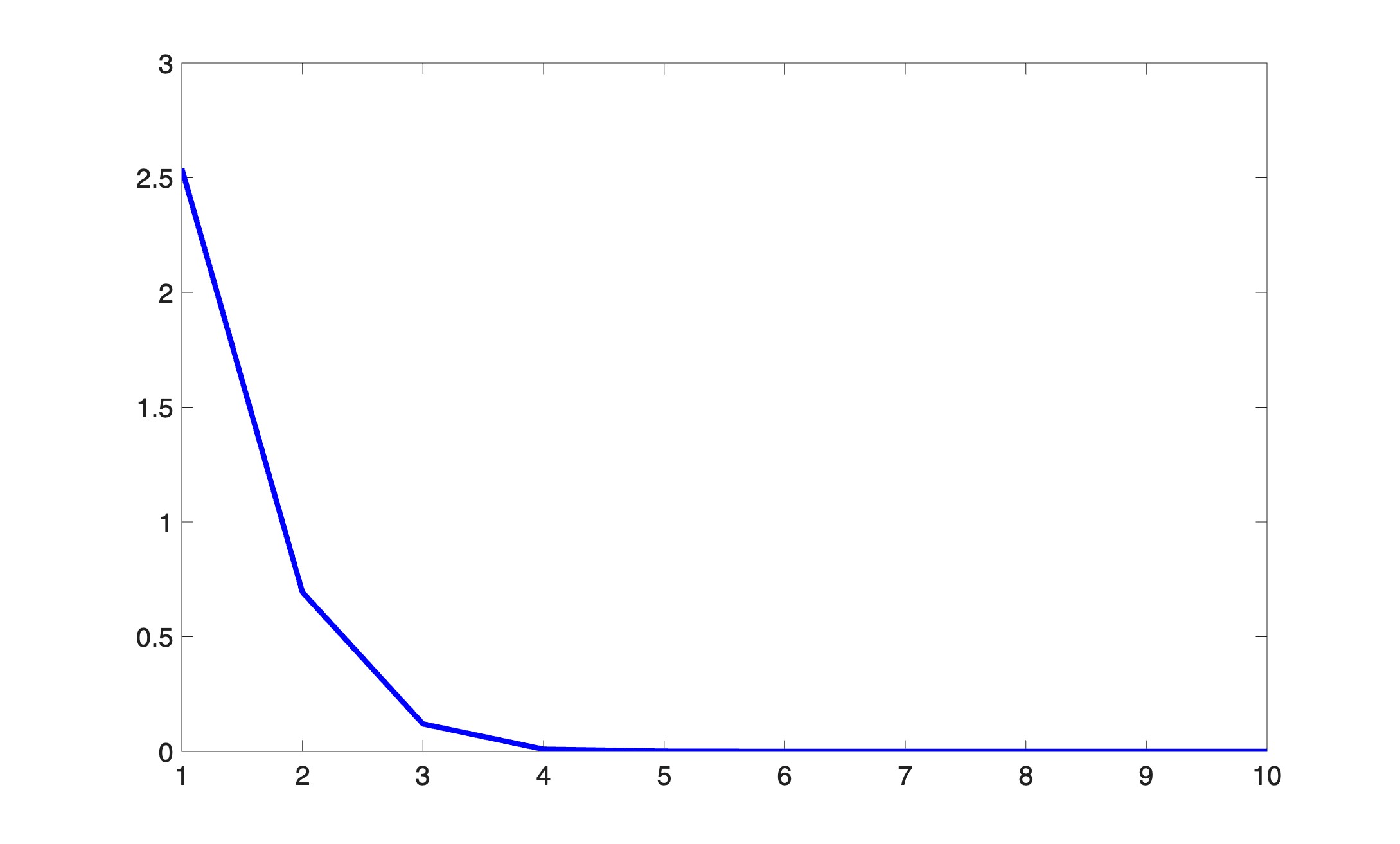}
}

\caption{Test 2. The first row shows the results for the $5\%$ noise level, while the second row shows the results for the $10\%$ noise level. In each row, the left plot compares the true initial density and its reconstruction: the solid curve denotes the true function, and the dashed curve denotes the reconstructed one. The right plot shows the absolute consecutive error $\|f^{\mathrm{rec},(k+1)}-f^{\mathrm{rec},(k)}\|_{L^\infty(0,L)}$.}
\label{fig:test2}
\end{figure}

The reconstructions shown in Figure~\ref{fig:test2} demonstrate that the proposed method performs well for this smooth Gaussian initial density. For both noise levels, the reconstructed function accurately captures the location, width, and overall shape of the true initial density. To quantify the reconstruction accuracy, we use the relative $L^2$ error $\frac{\|f^{0,\mathrm{true}}-f^{0,\mathrm{rec}}\|_{L^2(0,L)}}{\|f^{0,\mathrm{true}}\|_{L^2(0,L)}}$ and the relative $L^\infty$ error $\frac{\|f^{0,\mathrm{true}}-f^{0,\mathrm{rec}}\|_{L^\infty(0,L)}}{\|f^{0,\mathrm{true}}\|_{L^\infty(0,L)}}$. For the $5\%$ noise level, these errors are $0.0420$ and $0.0441$, respectively. For the $10\%$ noise level, they are $0.0612$ and $0.0728$, respectively. These results indicate that the proposed method yields accurate reconstructions and maintains good robustness under noisy data. Moreover, the absolute consecutive errors decrease rapidly over iterations, confirming the stable numerical behavior of the Carleman--Picard scheme.

{\bf Test 3:} In Test 3, the true initial density is chosen as the probability density function of the uniform distribution
\[
f^0(v)=
\begin{cases}
2.5, & 0.6\le v\le 1,\\[1mm]
0, & \text{otherwise}.
\end{cases}
\]
This function is compactly supported on $[0.6,1]$ and has jump discontinuities at $v=0.6$ and $v=1$, where it changes abruptly between $0$ and $2.5$. Therefore, this example is more challenging than the previous two tests, since the target function is discontinuous.

\begin{figure}[ht]
\centering
\subfloat[Reconstruction of $f^0$ with $5\%$ noise]{
    \includegraphics[width=0.47\textwidth]{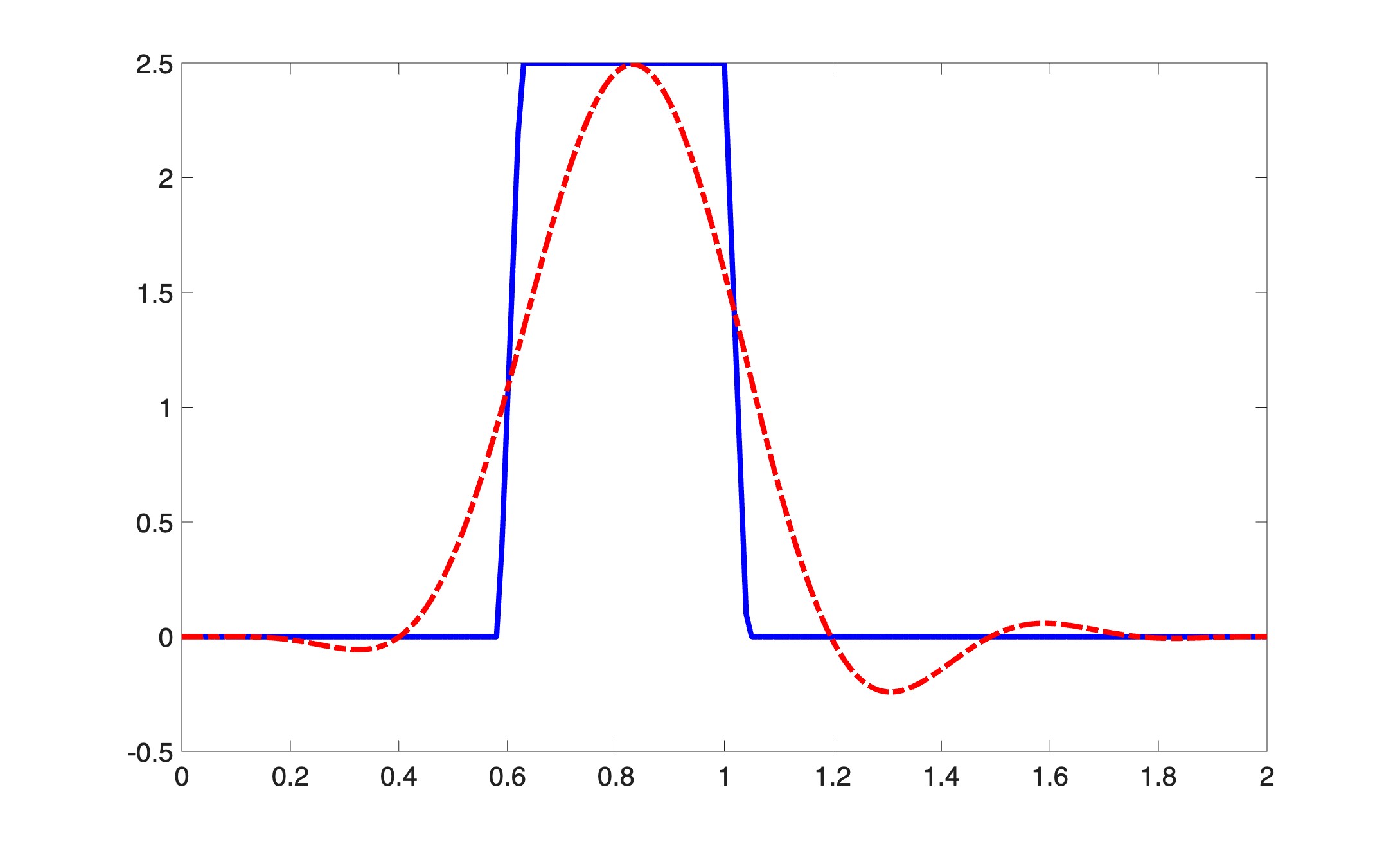}
}
\quad
\subfloat[Absolute consecutive errors with $5\%$ noise]{
    \includegraphics[width=0.47\textwidth]{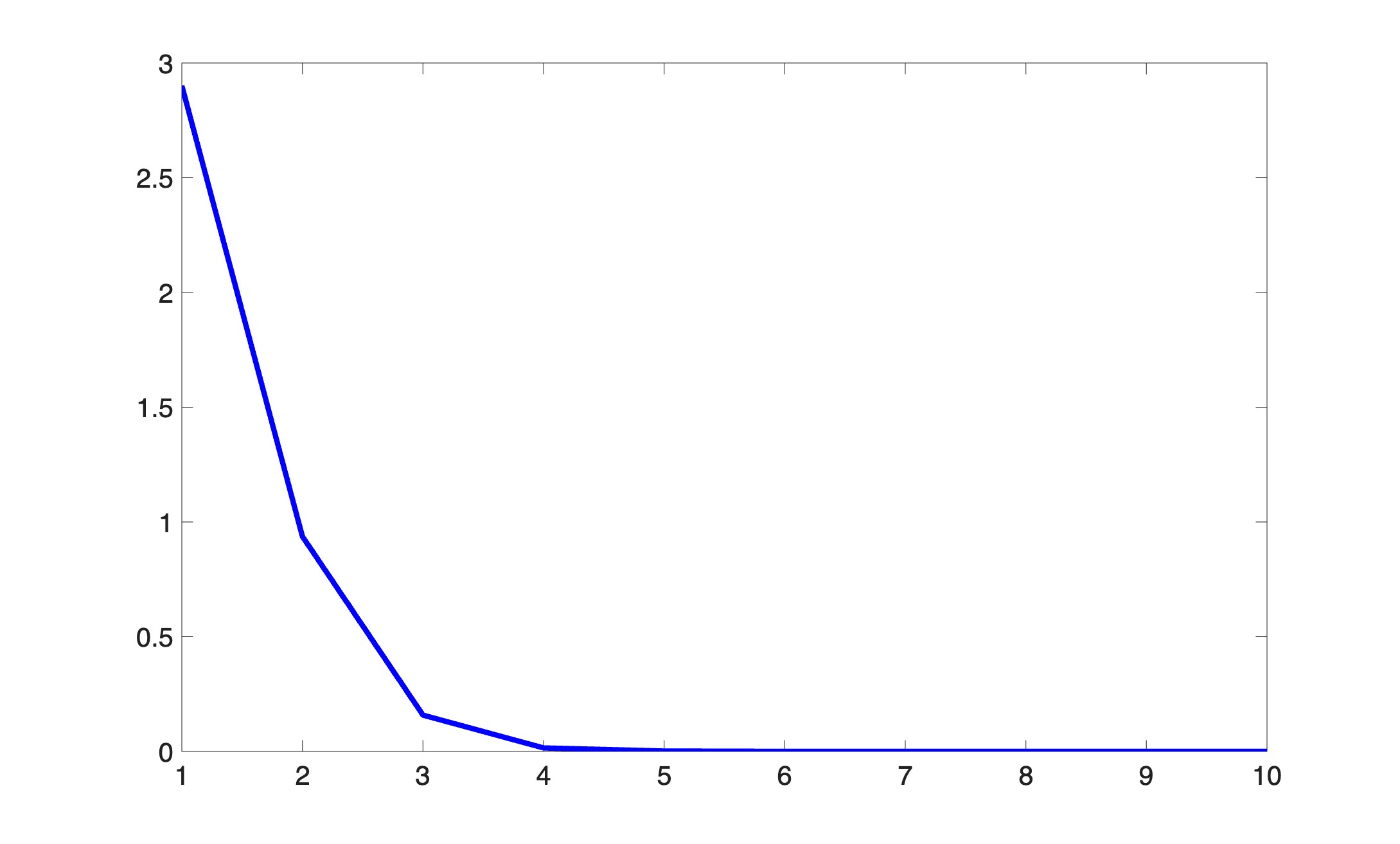}
}

\subfloat[Reconstruction of $f^0$ with $10\%$ noise]{
    \includegraphics[width=0.47\textwidth]{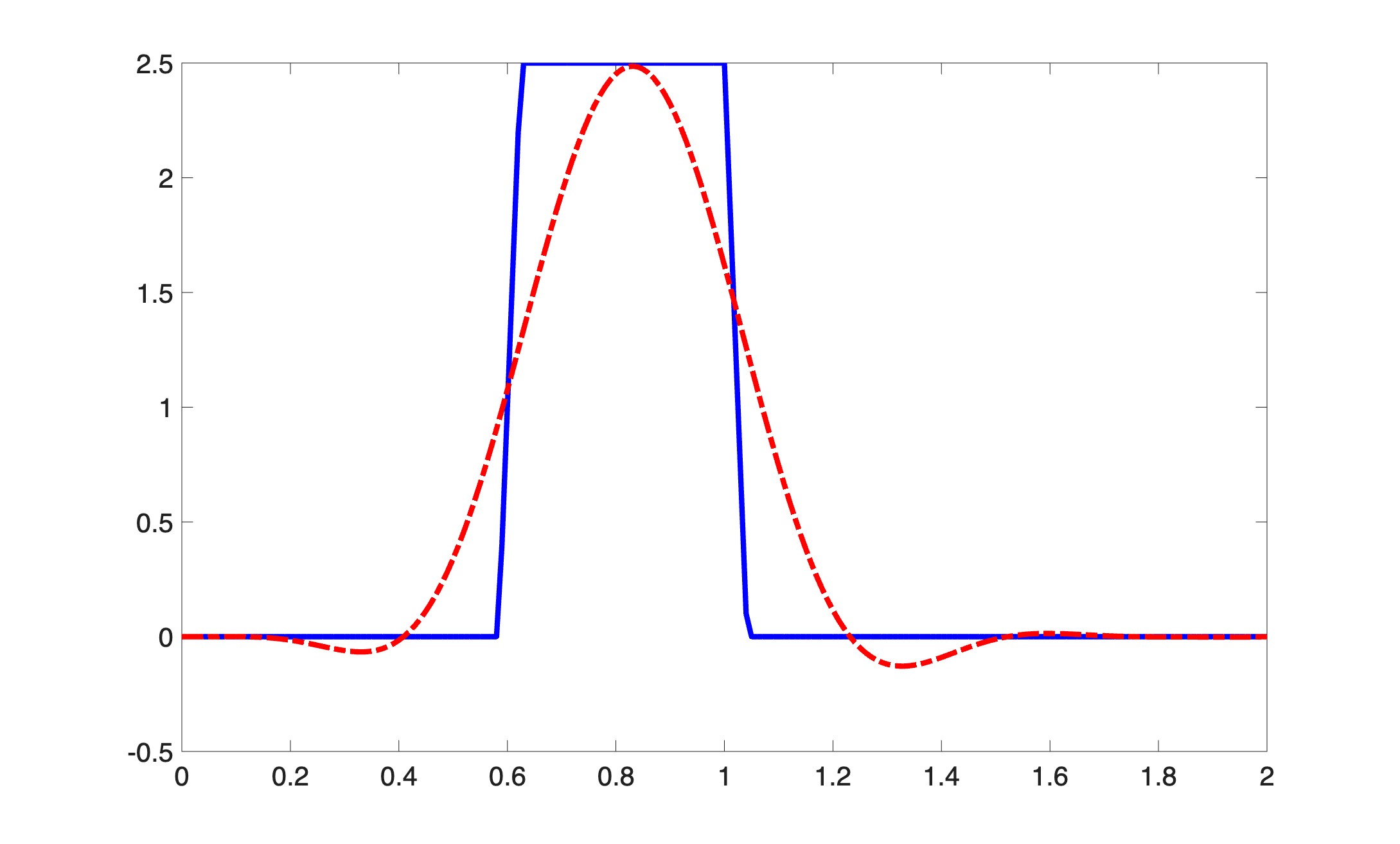}
}
\quad
\subfloat[Absolute consecutive errors with $10\%$ noise]{
    \includegraphics[width=0.47\textwidth]{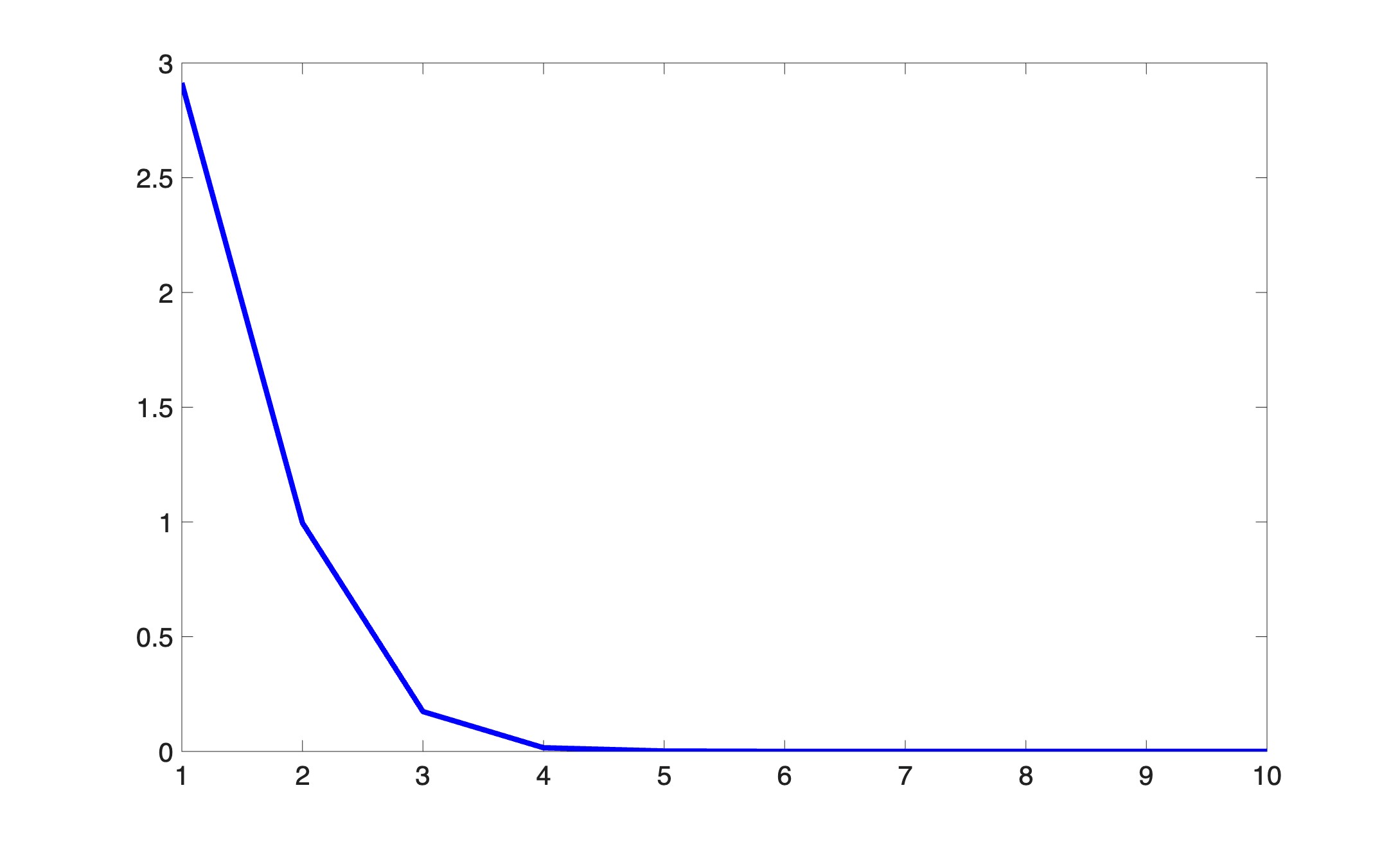}
}

\caption{Test 3. The first row shows the results for the $5\%$ noise level, while the second row shows the results for the $10\%$ noise level. In each row, the left plot compares the true initial density and its reconstruction: the solid curve denotes the true function, and the dashed curve denotes the reconstructed one. The right plot shows the absolute consecutive error $\|f^{\mathrm{rec},(k+1)}-f^{\mathrm{rec},(k)}\|_{L^\infty(0,L)}$.}
\label{fig:test3}
\end{figure}

The reconstructions in Figure~\ref{fig:test3} show that the proposed method remains effective even for this discontinuous initial density. Although the relative $L^2$ and $L^\infty$ errors are larger than those in the previous smooth tests, the reconstructed function still captures the main qualitative features of the true initial density, including the location of its support, the approximate height of the plateau, and the overall block-type structure. This behavior is reasonable because the true solution exhibits jump discontinuities, whereas the reconstruction procedure is based on a regularized least-squares formulation and therefore tends to produce smoother approximations near discontinuities. As a result, some smearing and oscillation near the jump locations are expected. 

To quantify the reconstruction accuracy, we use the relative $L^2$ error $\frac{\|f^{0,\mathrm{true}}-f^{0,\mathrm{rec}}\|_{L^2(0,L)}}{\|f^{0,\mathrm{true}}\|_{L^2(0,L)}}$ and the relative $L^\infty$ error $\frac{\|f^{0,\mathrm{true}}-f^{0,\mathrm{rec}}\|_{L^\infty(0,L)}}{\|f^{0,\mathrm{true}}\|_{L^\infty(0,L)}}$. For the $5\%$ noise level, these errors are $0.3106$ and $0.4631$, respectively. For the $10\%$ noise level, they are $0.3154$ and $0.4692$, respectively. The fact that these values change only slightly when the noise level increases from $5\%$ to $10\%$ indicates a certain degree of robustness of the method with respect to noise, even in this more challenging nonsmooth setting. Moreover, the absolute consecutive errors decrease rapidly over iterations, confirming the stable numerical behavior of the Carleman--Picard scheme.

{\bf Test 4:} In Test 4, the true initial density is chosen as the probability density function of a scaled Beta distribution:
\[
f^0(v)=
\begin{cases}
\dfrac{1}{2B(3,7)}\left(\dfrac{v}{2}\right)^2\left(1-\dfrac{v}{2}\right)^6,
& 0\le v\le 2,\\[2mm]
0, & v>2,
\end{cases}
\]
where $B(3,7)$ denotes the Beta function evaluated at $(3,7)$, that is,
\[
B(3,7)=\int_0^1 t^2(1-t)^6\,dt.
\]
This function is nonnegative and smooth on $(0,2)$, and it vanishes at both endpoints $v=0$ and $v=2$.

\begin{figure}[ht]
\centering
\subfloat[Reconstruction of $f^0$ with $5\%$ noise]{
    \includegraphics[width=0.47\textwidth]{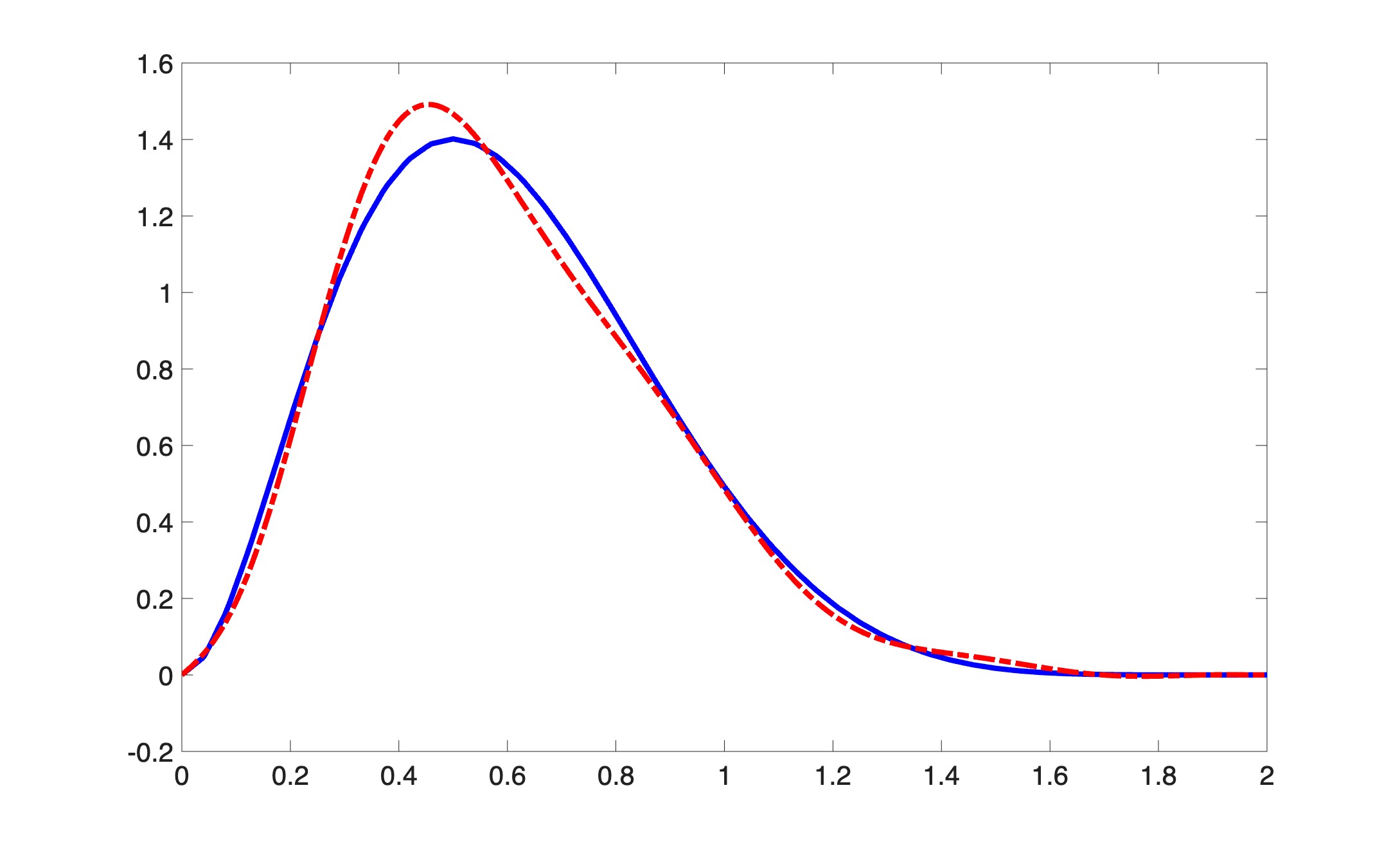}
}
\quad
\subfloat[Absolute consecutive errors with $5\%$ noise]{
    \includegraphics[width=0.47\textwidth]{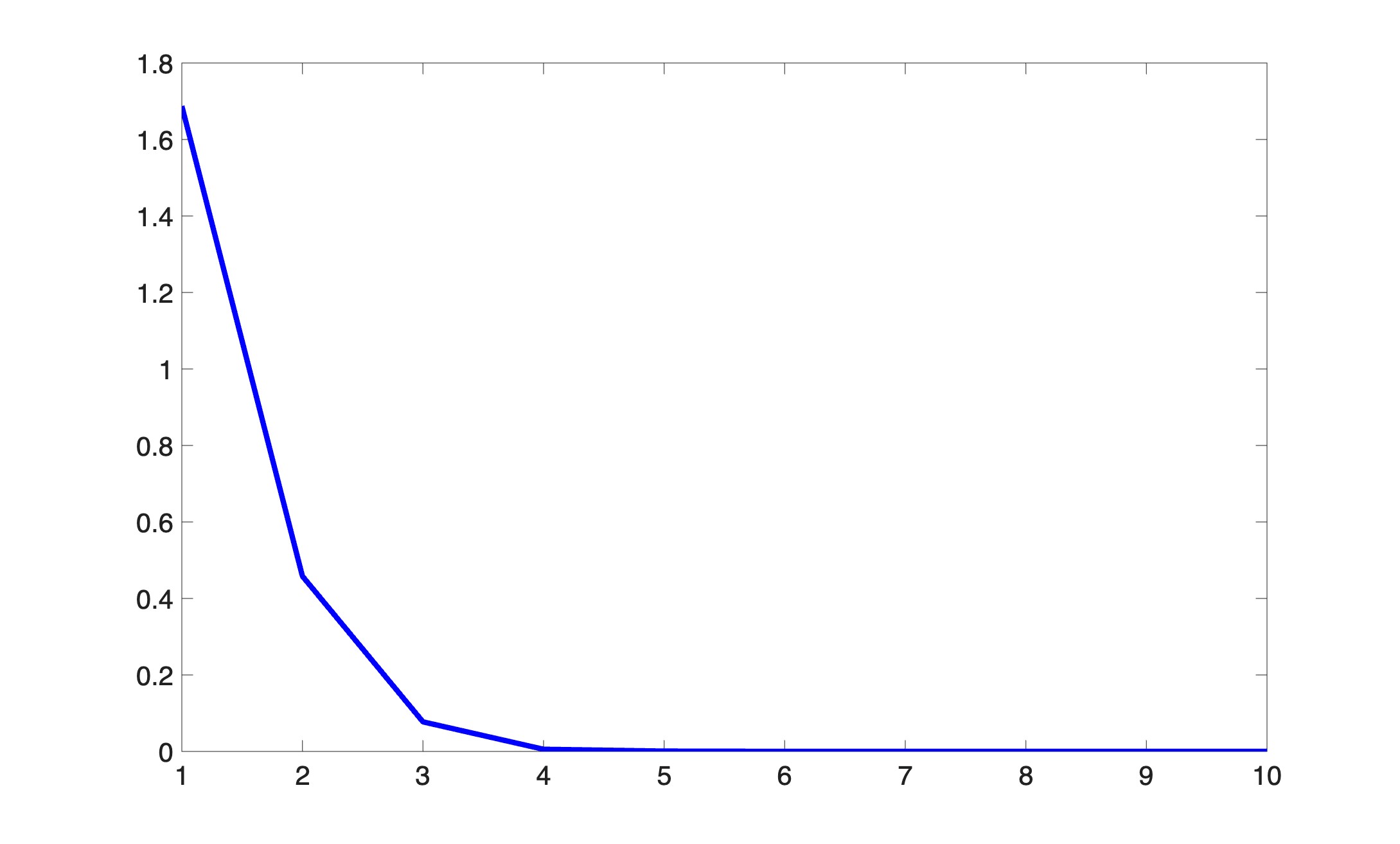}
}

\subfloat[Reconstruction of $f^0$ with $10\%$ noise]{
    \includegraphics[width=0.47\textwidth]{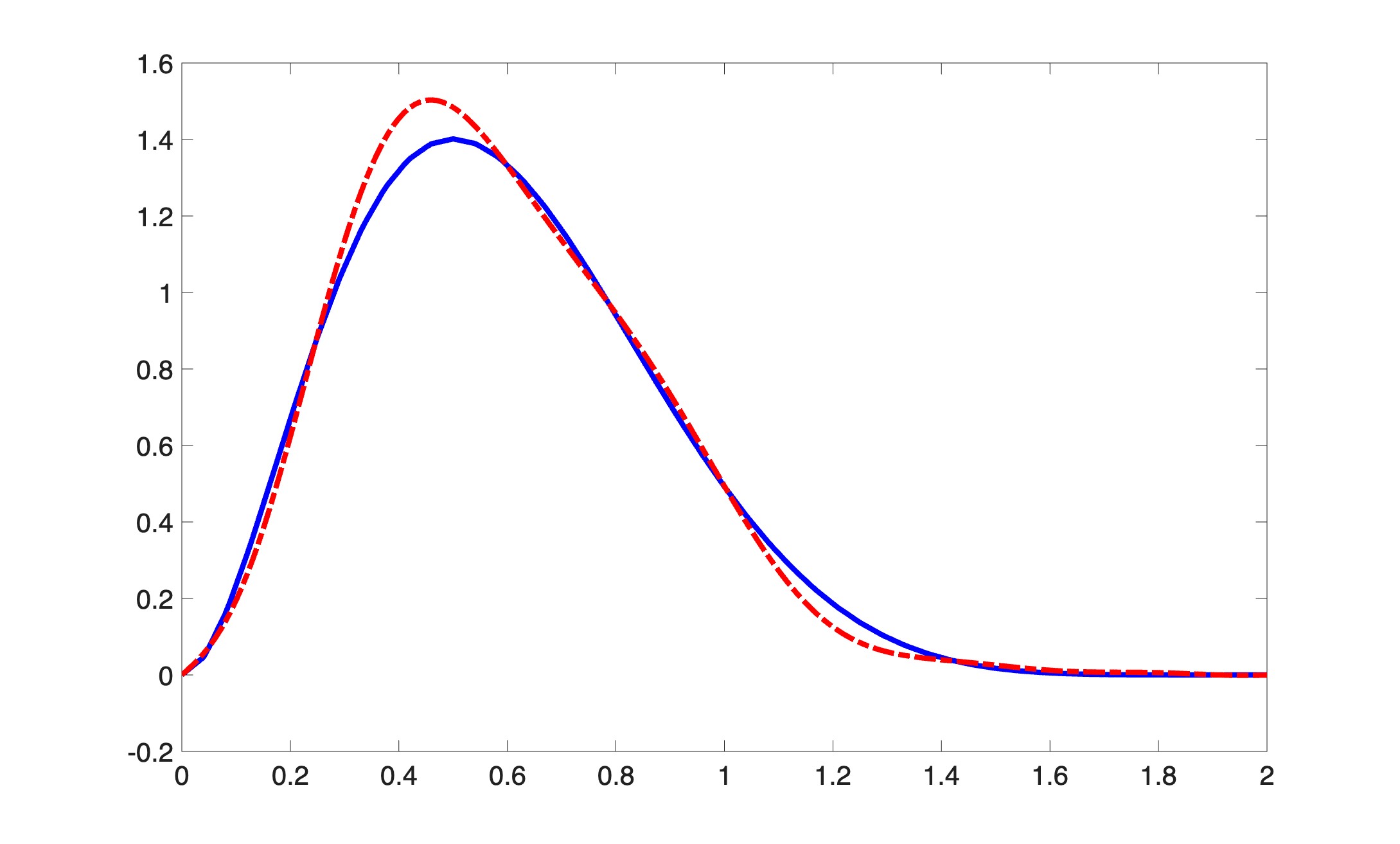}
}
\quad
\subfloat[Absolute consecutive errors with $10\%$ noise]{
    \includegraphics[width=0.47\textwidth]{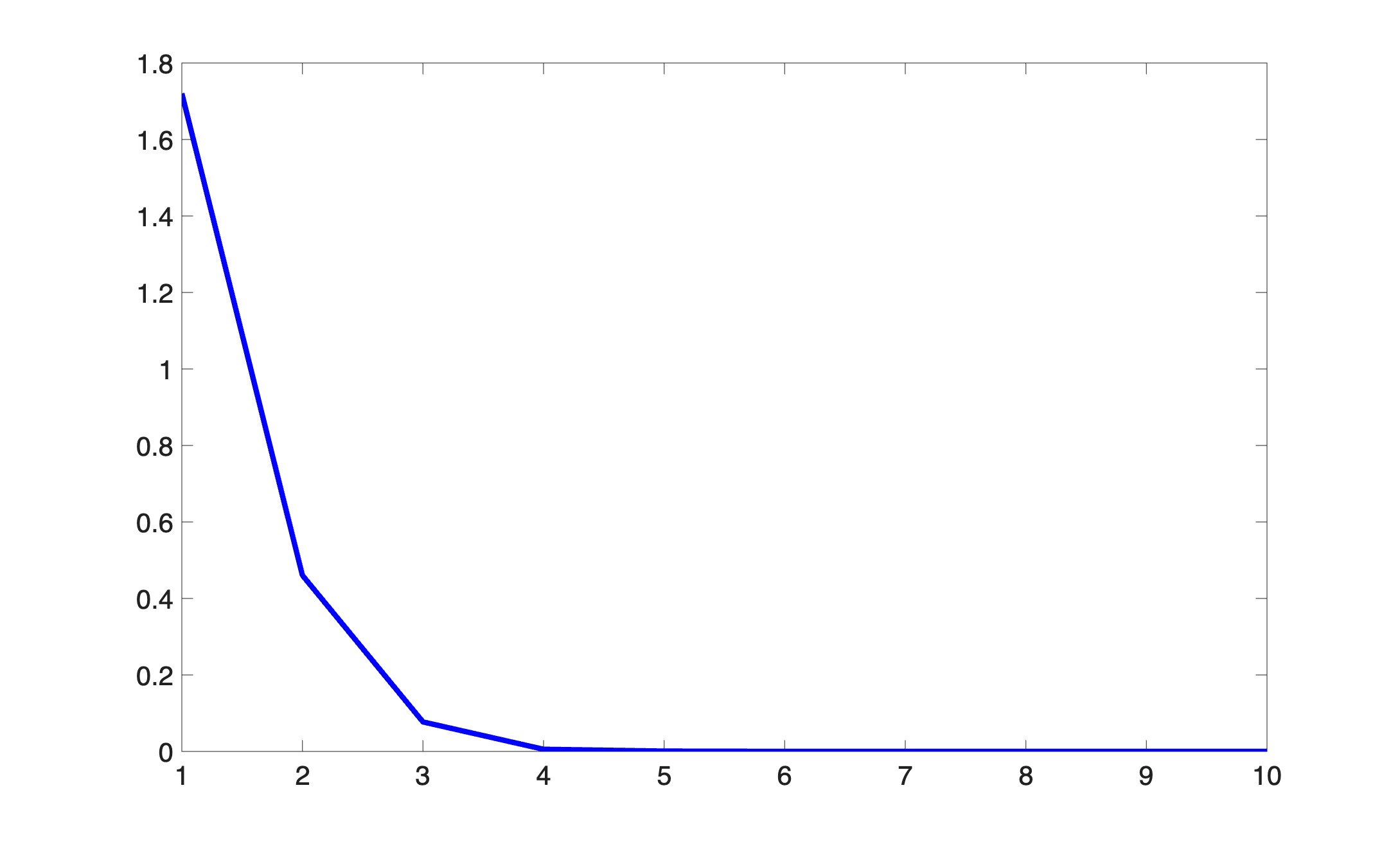}
}

\caption{Test 4. The first row shows the results for the $5\%$ noise level, while the second row shows the results for the $10\%$ noise level. In each row, the left plot compares the true initial density and its reconstruction: the solid curve denotes the true function, and the dashed curve denotes the reconstructed one. The right plot shows the absolute consecutive error $\|f^{\mathrm{rec},(k+1)}-f^{\mathrm{rec},(k)}\|_{L^\infty(0,L)}$.}
\label{fig:test4}
\end{figure}

The reconstructions in Figure~\ref{fig:test4} show that the proposed method performs well for this scaled Beta-type initial density. For both noise levels, the reconstructed function accurately reproduces the main characteristics of the true initial density, including the peak location, the overall asymmetric shape, and the decay toward both endpoints. To quantify the reconstruction accuracy, we use the relative $L^2$ error $\frac{\|f^{0,\mathrm{true}}-f^{0,\mathrm{rec}}\|_{L^2(0,L)}}{\|f^{0,\mathrm{true}}\|_{L^2(0,L)}}$ and the relative $L^\infty$ error $\frac{\|f^{0,\mathrm{true}}-f^{0,\mathrm{rec}}\|_{L^\infty(0,L)}}{\|f^{0,\mathrm{true}}\|_{L^\infty(0,L)}}$. For the $5\%$ noise level, these errors are $0.0657$ and $0.0930$, respectively. For the $10\%$ noise level, they are $0.0649$ and $0.0982$, respectively. These values indicate that the proposed method yields accurate reconstructions and remains stable under noisy data. Moreover, the absolute consecutive errors decrease rapidly over iterations, confirming the stable numerical behavior of the Carleman--Picard scheme.

\begin{Remark}
The numerical results provide clear evidence of the fast convergence of the Carleman--Picard iteration. In particular, the consecutive errors decrease rapidly as the iteration number increases; see, for example, Figures~\ref{fig1b} and~\ref{fig1d}. This decay is consistent with the geometric rate $\rho^k$ predicted by Theorem~\ref{thm:carleman_picard_convergence}, where $\rho\in(0,1)$. Hence, the numerical experiments confirm the global convergence analysis and show that the proposed method converges rapidly in practical computations.
\end{Remark}





\begin{Remark}
Although the objective of the inverse problem is to recover the initial density \(f(v,0)\), our method actually yields an approximation of the full function \(f(v,t)\) for all \((v,t)\in [0, L]\times[0, T]\). Indeed, once the coefficient functions are reconstructed, the function \(f^{\mathrm{rec}}(v,t)\) can be computed for every \((v,t)\in [0, L]\times[0, T]\) using the approximation formula in Step~\ref{s8} of Algorithm~\ref{alg:carleman_picard}.
Therefore, the numerical algorithm recovers not only the initial profile but also the entire time-dependent solution on the reconstruction domain.

In this paper, however, we report mainly the profile \(f(v,0)\), since this is the unknown quantity of principal interest in the inverse problem. For illustration, in Figure \ref{fig:fullTest3}, we include in Test~3 the reconstructed full field \(f^{\mathrm{rec}}(v,t)\), the corresponding exact solution \(f^{\mathrm{true}}(v,t)\), and the pointwise relative error, which is computed as 
\[ \mathcal{E}(v,t) = \frac{|f^{\mathrm{true}}-f^{\mathrm{rec}}|}{\|f^{\mathrm{true}}\|_{L^{\infty}((0, L) \times (0, T))}},\]
on \([0, L]\times[0 ,T]\). These additional plots show that the computed solution agrees well with the true one over the entire space-time domain, with larger errors concentrated only in a few limited regions.

\begin{figure}[ht]
\centering
\subfloat[$f^{true}$]{
    \includegraphics[width=0.3\textwidth]{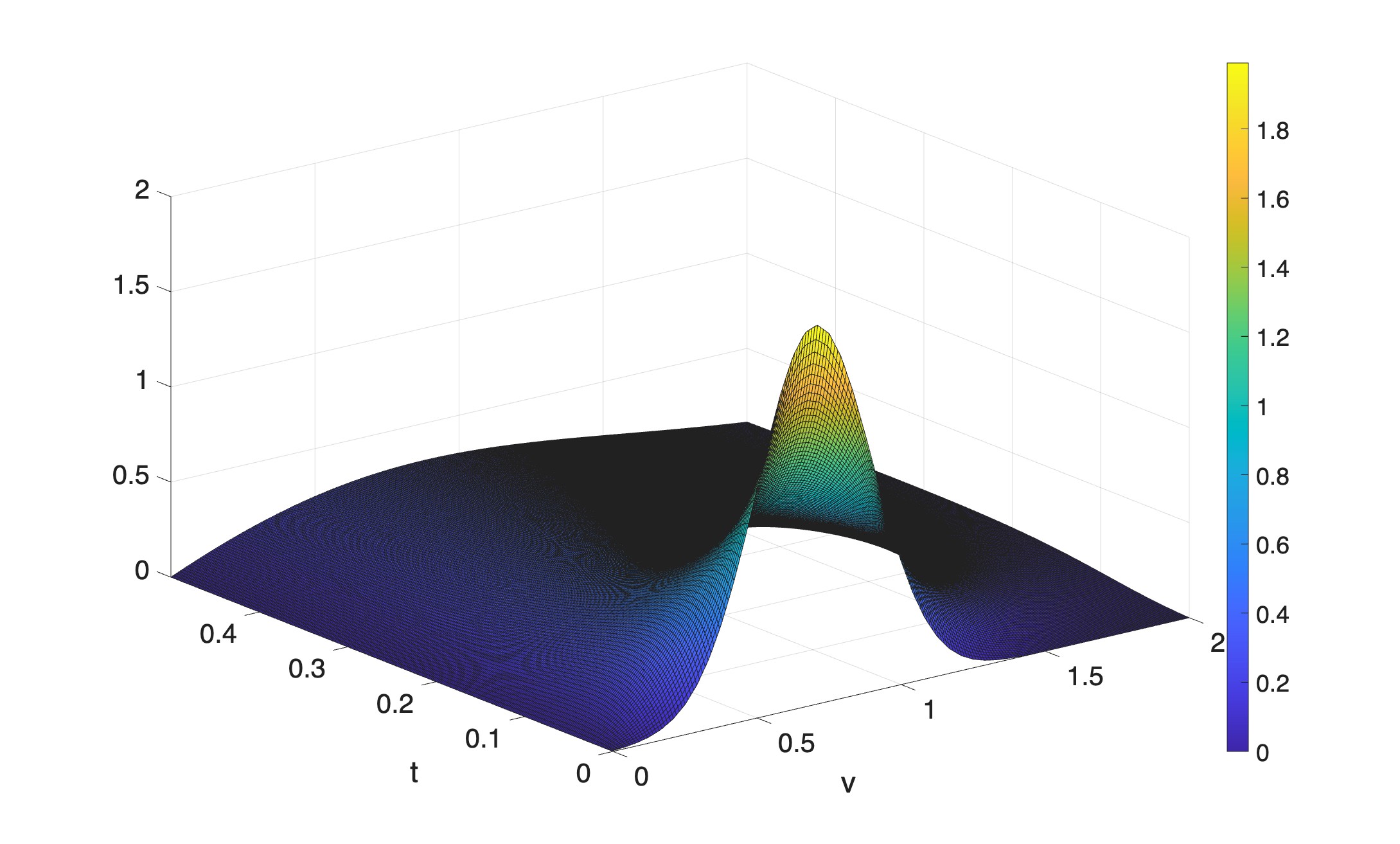}
}
\quad 
\subfloat[\(f^{\mathrm{rec}}\) obtained from data with \(5\%\) noise]{
    \includegraphics[width=0.3\textwidth]{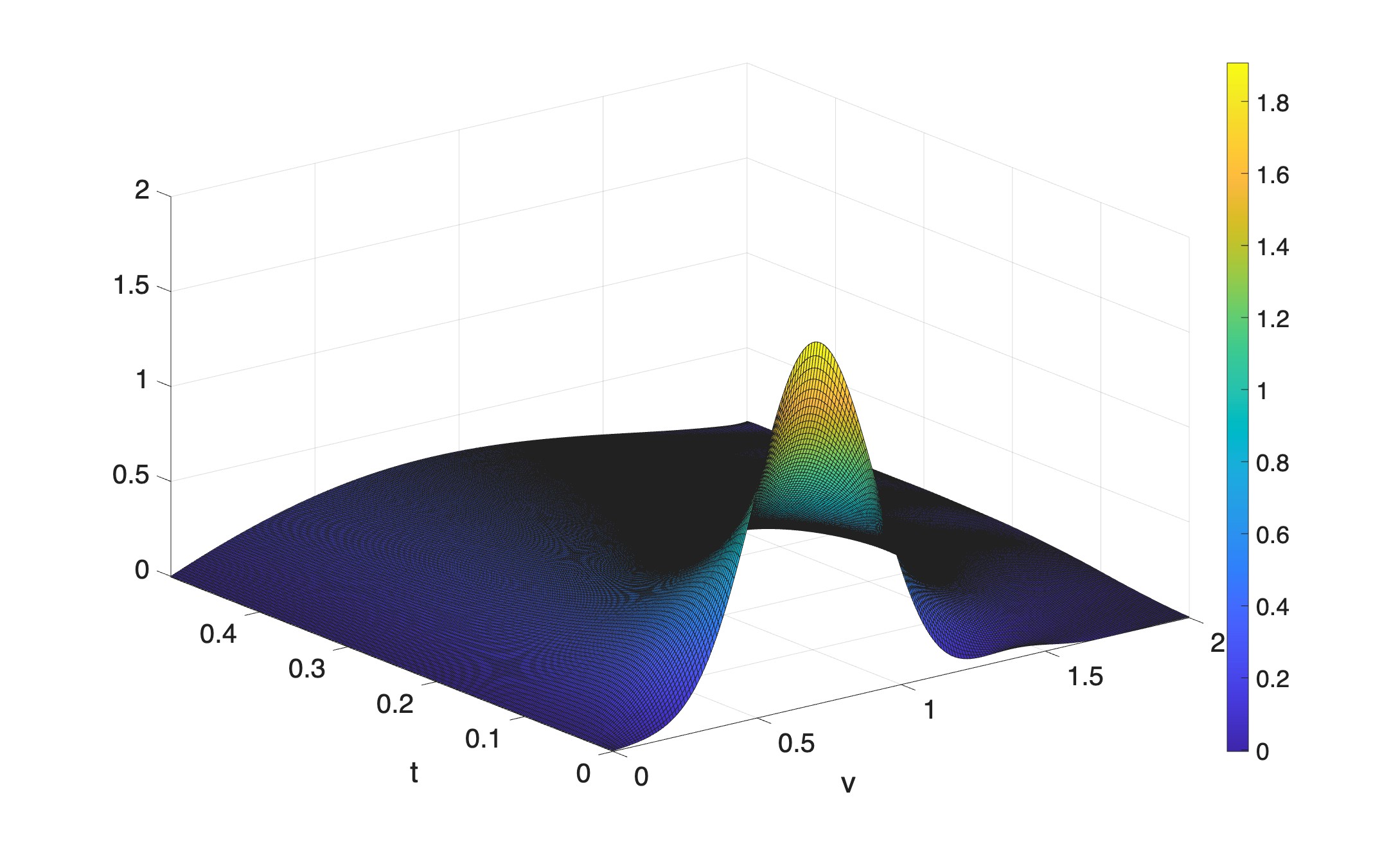}
}
\quad
\subfloat[Pointwise relative error $\mathcal{E}(v,t)$]{
    \includegraphics[width=0.3\textwidth]{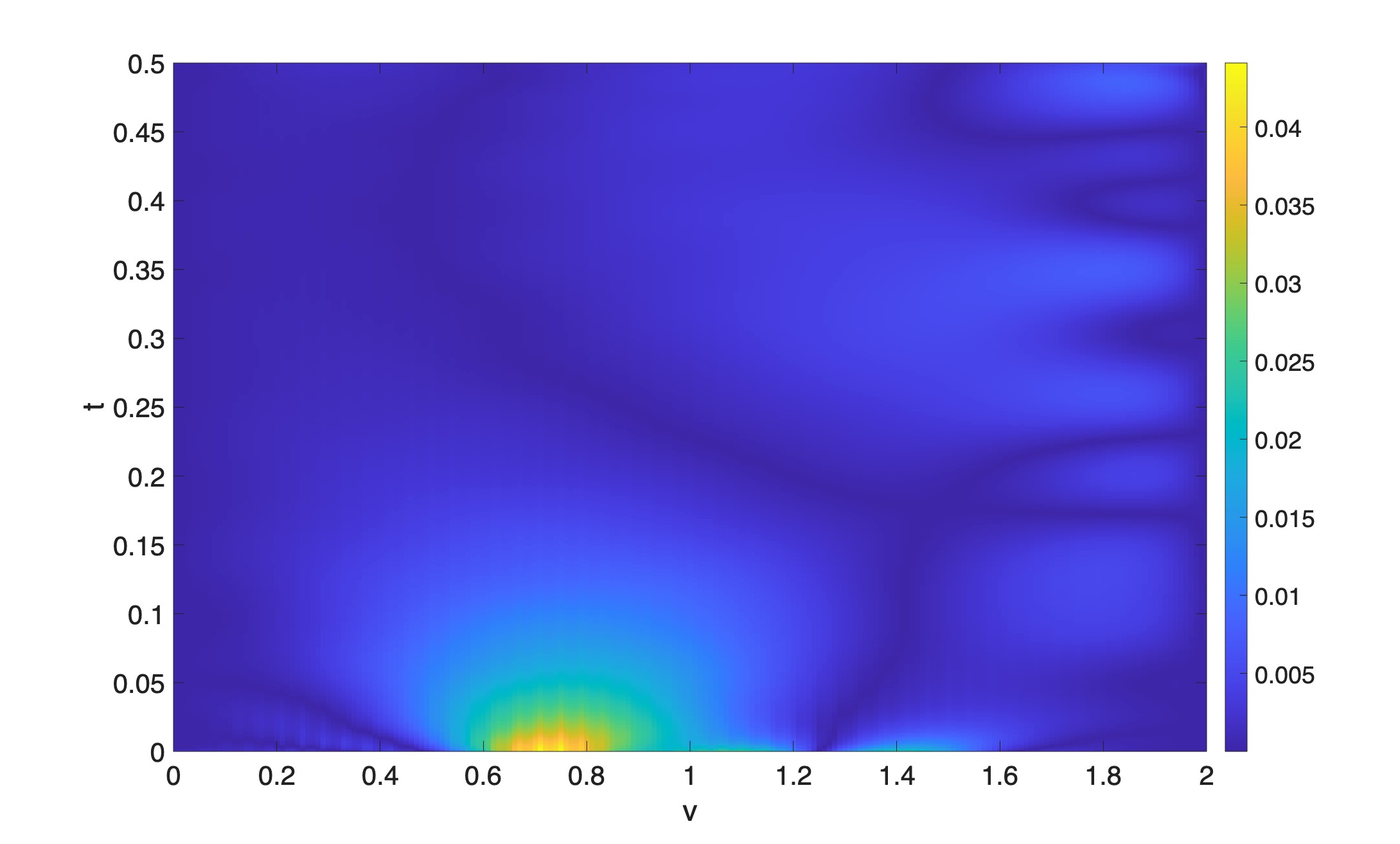}
}
\caption{\label{fig:fullTest3}
Test~3. Full space-time reconstruction on \([0,L]\times[0,T]\): 
(a) exact solution \(f^{\mathrm{true}}(v,t)\), 
(b) reconstructed solution \(f^{\mathrm{rec}}(v,t)\) from data with $5\%$ noise, and 
(c) pointwise relative error $\mathcal{E}(v,t)$. 
The reconstructed surface is in good agreement with the true one over most of the domain.}
\end{figure}
\end{Remark}

\section{Concluding remarks}\label{sec6}

In this paper, we investigated an inverse initial-density problem for a coagulation--fragmentation equation with size convection-diffusion. The objective was to reconstruct the unknown initial particle-size distribution from time-dependent boundary observations of the solution and its size derivative. To solve this problem, we developed a globally convergent reconstruction method that combines a Legendre--exponential time reduction with a Carleman--Picard iterative procedure.

The method first eliminates the time variable by projecting the solution onto a truncated polynomial--exponential basis, thereby reducing the original inverse problem to a coupled system for the spatial--mode coefficients. This reduced nonlinear system is then solved by a Carleman-weighted Picard iteration. The global convergence of the method is ensured by the Carleman weight and the associated Carleman estimate. At the same time, truncating the Fourier expansion improves stability by filtering out highly oscillatory noise components.

Under the assumptions imposed in this paper, we established a rigorous convergence result for the Carleman--Picard iteration and obtained a complete reconstruction procedure for the unknown initial density. The numerical experiments confirm the theoretical analysis and show that the proposed method yields accurate and stable reconstructions for several representative examples, including both smooth and nonsmooth initial profiles and noisy boundary data.

\section*{Acknowledgment}

M.-B. Tran is funded in part by NSF CAREER DMS-2303146 and NSF Grants DMS-2305523 and DMS-2306379.

\bibliographystyle{plain}
\bibliography{References1}
\end{document}